\documentclass[12pt]{amsart}
\usepackage{amsthm}
\newtheorem{theorem}{Theorem}[section]
\newtheorem{lemma}[theorem]{Lemma}
\newtheorem{proposition}[theorem]{Proposition}
\theoremstyle{definition}
\newtheorem{definition}[theorem]{Definition}

\newtheorem{example}[theorem]{Example}
\theoremstyle{remark}
\newtheorem{remark}[theorem]{Remark}
\newtheorem{corollary}[theorem]{Corollary}

\counterwithin*{section}{part}

\usepackage{amssymb}
\usepackage{empheq}
\usepackage{stmaryrd}
\usepackage{enumerate}
\usepackage[shortlabels]{enumitem}
\usepackage{calc}
\usepackage{url}
\usepackage{comment}

\newcommand{\norm}[1]{\left\lVert#1\right\rVert}

\def\EE{\mathbb{E}}
\def\NN{\mathbb{N}}
\def\PP{\mathbb{P}}

\newcommand{\CAT}{{\rm{CAT}(0)}}

\usepackage[left=2.0cm,%
                right=2.0cm,%
                top=2.5cm,%
                bottom=3.5cm,%
                headheight=12pt,%
                a4paper]{geometry}%

\begin{document}

\title[Convergence guarantees for stochastic algorithms in metric spaces]{Convergence guarantees for stochastic algorithms solving non-unique problems in metric spaces}

\author[N. Pischke and T. Powell]{Nicholas Pischke and Thomas Powell}
\date{\today}
\thanks{{\bf Acknowledgements:} The authors want to thank Morenikeji Neri and Jacqueline Treusch for their comments on a previous version, and the anonymous referees for the many helpful suggestions which improved the paper in many places.\\
{\bf Funding:} The second author was partially supported by the EPSRC grant EP/W035847/1.\\
{\bf Conflicts of interests:} The authors have no relevant financial or non-financial interests to disclose.\\
{\bf Data availability:} Data sharing not applicable to this article as no data sets were generated or analyzed.
}
\maketitle
\vspace*{-5mm}
\begin{center}
{\scriptsize 
Department of Computer Science, University of Bath,\\
Claverton Down, Bath, BA2 7AY, United Kingdom,\\
E-mails: $\{$nnp39,trjp20$\}$@bath.ac.uk}
\end{center}

\maketitle
\begin{abstract}
We prove a general quantitative theorem on the asymptotic behavior of stochastic quasi-Fej\'er monotone sequences in a broad metric context.  Concretely, our result explicitly constructs a rate of convergence for such process, both in mean and almost surely, under an abstract stochastic regularity assumption, derived from previous work of Kohlenbach, L\'opez-Acedo and Nicolae [Isr.\ J.\ Math.\ 232(1), pp.\ 261-297, 2019] on such notions in a deterministic context. Our notion of regularity extends and unifies many common conditions from the literature, such as generalized contractivity for self maps, weak sharp minima and error bounds for real-valued functions, uniform monotonicity and global metric subregularity for set-valued operators, related Polyak-{\L}ojasiewicz or Kurdyka-{\L}ojasiewicz conditions, as well as expected sharp growth as e.g.\ studied by Asi and Duchi [SIAM J.\ Optim.\ 29(3), pp.\ 2257-2290, 2019]. The rate is moreover highly uniform, depending only on very few data of the surrounding objects. We also discuss special cases which allow for the construction of fast rates in the form of linear non-asymptotic guarantees. We conclude by presenting three concrete methods from stochastic approximation where our results yield new rates of convergence, including the classical example of the stochastic proximal point method, a randomized variant of the Krasnoselskii-Mann scheme for solving stochastic fixed-point equations, and a Busemann subgradient method recently introduced by Goodwin, Lewis, L\'opez-Acedo and Nicolae [Math.\ Program., to appear], all of which make use of our metric generality by being formulated over complete geodesic metric spaces of nonpositive curvature.
\end{abstract}

\section{Introduction}

\subsection{The setup}

We are concerned with the study of stochastic algorithms $(x_n)$ that approximate almost-sure solutions to the abstract problem
\[
\text{find a zero }z\in\mathrm{zer}F:=\{z\in X\mid F(z)=0\}\text{ for a function }F:X\to [0,\infty],
\]
over arbitrary separable and complete metric spaces $(X,d)$. A wide range of deterministic and stochastic problems can be naturally brought into this form, including almost-sure fixed point problems or mean minimization problems (see Section \ref{sec:examples} later on).

As with the problem formulation, we consider a broad class of algorithms: Instead of confining ourselves to a specific iteration schema, we consider arbitrary stochastic processes $(x_n)$, adapted to a filtration $(\mathsf{F}_n)$ over a probability space $(\Omega,\mathsf{F},\PP)$, that are stochastically quasi-Fej\'er monotone, i.e.
\[
\EE[d(z,x_{n+1})\mid\mathsf{F}_n]\leq (1+\zeta_n)d(z,x_n)+\xi_n\text{ a.s.}
\]
for all $n\in\mathbb{N}$ and all $z\in\mathrm{zer}F$ and where $(\zeta_n),(\xi_n)$ are summable nonnegative random variables adapted to $(\mathsf{F}_n)$. This represents a relaxed supermartingale condition in line with similar abstract approaches in the literature.

The driving question of our paper is concerned with quantitative aspects of convergence:\\

\noindent\textbf{Fundamental question:} When can effective rates towards solutions, in mean and almost surely, be explicitly described, for general classes of problems $F$ and algorithms $(x_n)$?\\

To make the above question more precise, we have to specify the exact notions of rates that are being considered in this paper, as well as how ``effective'' is to be understood.

For the former, given a stochastic process $(x_n)$ as above, a rate of convergence in mean for $(x_n)$ towards another random variable $x$ is a function $\rho:(0,\infty)\to\mathbb{N}$ such that
\[
\forall\varepsilon>0\ \forall n\geq\rho(\varepsilon)\left( \EE[d(x_n,x)]<\varepsilon\right).
\]
This formulation is a much more explicit type of quantitative estimate than what is often understood under a rate of convergence, with a particular benefit being that if one is given such a $\rho$ which is also invertible (which can often be guaranteed), then the above property immediately implies the non-asymptotic estimate $\EE[d(x_n,x)]<\rho^{-1}(n)$ for all $n\in\mathbb{N}$, which can also be used to provide an asymptotic complexity bound on the sequence in terms of the commonly used big $\mathcal{O}$ notation, entailing $(\EE[d(x_n,x)])=\mathcal{O}(\rho^{-1}(n))$. 

An almost sure rate of convergence for $(x_n)$ towards $x$ on the other hand is a function $P:(0,\infty)^2\to\mathbb{N}$ such that
\[
\forall \lambda,\varepsilon>0\left( \PP\left(\exists m\geq P(\lambda,\varepsilon)\, (d(x_m,x)\geq \varepsilon)\right)<\lambda\right).
\]
Also this formulation provides a very explicit quantitative estimate for the convergence behaviour of $(x_n)$, in that if $P_\varepsilon:=P(\cdot,\varepsilon)$ is invertible for any $\varepsilon>0$ (which, similarly to the above, can often be guaranteed), it also entails a non-asymptotic estimate in the form of $\PP\left(\exists m\geq n\, (d(x_m,x)\geq \varepsilon)\right)<P_\varepsilon^{-1}(n)$ for all $n\in\mathbb{N}$ and $\varepsilon>0$, as before in particular providing a bound in big $\mathcal{O}$ notation with $(\PP\left(\exists m\geq n\, (d(x_m,x)\geq \varepsilon)\right))=\mathcal{O}(P^{-1}_\varepsilon(n))$. 

Rates $\rho:(0,\infty)\to\mathbb{N}$ and $P:(0,\infty)^2\to\mathbb{N}$ as above are then considered to be \emph{effective} if their restrictions to rational inputs, that is $\rho\vert_{(0,\infty)_\mathbb{Q}}$ and $P\vert_{(0,\infty)_\mathbb{Q}^2}$ where $(0,\infty)_\mathbb{Q}:=(0,\infty)\cap\mathbb{Q}$, are computable in the formal sense of computability theory and computable analysis (we refer e.g.\ to \cite{Cooper2004}), that is if there is a computer program (in the formal sense, see again \cite{Cooper2004}) which can evaluate the functions $\rho$ and $P$ on their rational values. While this is a very weak condition for ensuring that we can actually ``use'' the rates, it is well-known that even this weak requirement of being effective makes our fundamental question above nontrivial, since effective rates cannot always be found, a phenomenon that in optimization theory is sometimes referred to as \emph{arbitrary slow convergence}. This can be formally explained using tools from mathematical logic and computability theory, where so-called Specker sequences \cite{Specker1949} can be used to show that already in very simple deterministic cases, computable rates of convergence are generally unachievable for general quasi-Fej\'er monotone methods $(x_n)$ and problems $F$ (see in particular \cite{Neumann2015}).

\subsection{The theoretical contributions of the paper}

Inspired by the work of Kohlenbach, L\'opez-Acedo and Nicolae \cite{KohlenbachLopezAcedoNicolae2019},\footnote{The present paper is part of an ongoing effort, initiated in the last two years, to bring methods from proof mining \cite{Kohlenbach2008,Kohlenbach2019} systematically to bear on probability theory and stochastic optimization. The works that this paper builds on, notably \cite{KohlenbachLopezAcedoNicolae2019} and \cite{NeriPischkePowell2025} together with \cite{Pischke2025b} (as well as the related \cite{FreundPischke2026,NeriPischkePowell2026}) are similarly part of this proof-theoretic approach to computational results in mathematics. Here, for example, this logical approach in particular influenced the way in which we formulated the various (stochastic) moduli. The rest of the paper avoids any technical reference to mathematical logic.} we consider a broad class of problems that can be characterised, in an abstract way, by considering functions $F:X\to [0,\infty]$ that satisfy a generalized regularity assumption. For that, we extend the generalized notion of regularity introduced in \cite{KohlenbachLopezAcedoNicolae2019} here to a stochastic context, taking the form of a modulus $\tau:(0,\infty)\to (0,\infty)$ satisfying
\[
\forall \varepsilon>0\ \forall x\in D\left(\EE[F(x)]<\tau(\varepsilon)\to \EE[\mathrm{dist}_{\mathrm{zer}F}(x)]<\varepsilon\right)
\]
for a suitable collection of $X$-valued random variables $D$. As we will show, such regularity conditions are intimately related to associated growth conditions in mean
\[
\forall x\in D\left(\EE[F(x)]\geq \tau(\EE[\mathrm{dist}_{\mathrm{zer}F}(x)])\right)
\]
for the mapping $F$.

As we will discuss in detail (see Section \ref{sec:examples}), these conditions uniformly cover various notions such as: generalized contractions and retractions; generalized weak sharp minima in the sense of \cite{KohlenbachLopezAcedoNicolae2019} (extending \cite{BurkeFerris1993}, see also \cite{BurkeDeng2002,Ferris1991,LiMordukhovichWangYao2011}) and related to that corresponding notions of error bounds (see e.g.\ \cite{BolteNguyenPeypouquetSuter2017,DrusvyatskiyLewis2018}) as well as polynomial growth conditions (see e.g.\ \cite{Shapiro1994}) and expected sharp growth \cite{AsiDuchi2019}; generalized metric subregularity (see e.g.\  \cite{HermerLukeSturm2019,LiMordukhovichZhu2026}, extending \cite{Leventhal2009,DontchevRockafellar2009}) and related Polyak-{\L}ojasiewicz, or more generally Kurdyka-{\L}ojasiewicz conditions (see e.g.\ \cite{BolteDaniilidisLeyMazet2010,BolteNguyenPeypouquetSuter2017,Zhang2020}); uniform and strong accretivity and monotonicity for operators and vector fields; uniform and strong convexity for functions. We note that many of these conditions, along with our abstract regularity notion in general, do not necessarily entail that $\mathrm{zer}F$ has a unique solution.

It is known that already for the deterministic Picard iteration, such a (deterministic) modulus of regularity can be explicitly constructed from an associated (uniform) rate of convergence for the process towards a solution (see Proposition 4.4 in \cite{KohlenbachLopezAcedoNicolae2019}), illustrating that the presence of this generalized regularity condition is fundamentally connected to the existence of explicit (and suitably uniform) rates of convergence. This in particular justifies the wide generality of the approach not only in the deterministic but also in the stochastic setting.

Our main results are general quantitative convergence theorems for stochastically quasi-Fej\'er monotone iterations relative to solution sets of problems that satisfy such a stochastic notion of regularity, requiring in addition only a mild approximation property (see Theorems \ref{thm:main} and \ref{thm:mainWeak}). These results provide explicit constructions for rates of convergence, in mean and almost surely, which are moreover highly uniform, depending only on very few data of the surrounding objects, and in particular being independent of the distribution of the space. Crucially, whenever these input data are computable, the final rates will also be computable, so that our constructions are in particular effective. Furthermore, we provide a similarly general result on fast (that is in this case linear) nonasymptotic guarantees in the context of linear regularity and suitable assumptions on the surrounding parameters (see Theorem \ref{thm:fast}). Finally, all of these results are formulated in the even broader context of general distance functions $\phi:X\times X\to [0,\infty)$, following previous work of the first author in a deterministic context \cite{Pischke2025b} (and of the authors and Neri \cite{NeriPischkePowell2025} in certain stochastic settings, as discussed below), which allow us to uniformly cover perturbations of the metric and distance functions beyond that, such as Bregman distances.

Some initial results on effective convergence for stochastic quasi-Fej\'er monotone sequences are contained in the authors' earlier work \cite{NeriPischkePowell2025}, confined to the more restrictive class of problems with unique solutions. The present paper extends those results substantially: Along with a much more general and extensive treatment of stochastic regularity, we are confronted with multiple demanding technical difficulties that are entirely absent in the case of unique solutions. For example, to deal with non-unique problems, our main results combine a quantitative martingale argument motivated by \cite{NeriPischkePowell2025} with a new measurable selection argument along the filtration $(\mathsf{F}_n)$ of the process $(x_n)$. By nature of the filtration, these results from measurable selection theory cannot assume the completeness of the measure space and hence require additional care beyond that of canonical results. In that context, we are forced to consider a strengthened variant of stochastic quasi-Fej\'er monotonicity, introduced abstractly for the first time here, which allows one to maintain the supermartingale condition 
\[
\EE[d(z,x_{n+1})\mid\mathsf{F}_n]\leq (1+\zeta_n)d(z,x_n)+\xi_n\text{ a.s.}
\]
even for $\mathsf{F}_n$-measurable and suitably integrable $\mathrm{zer}F$-valued random variables $z$. This notion is often satisfied outright for many methods (in particular for the applications discussed in this paper), but we additionally examine this stronger property in relation to the usual formulation, providing a general lifting result for a class of distances including metric powers (abstracting recent work of Combettes and Madariaga \cite{CombettesMadariaga2025}).

\subsection{The applied contributions of the paper}

We provide three example applications of our theoretical work, yielding new results for both classical and recently introduced methods. To illustrate the metric generality of our framework, all applications are set in the context of Hadamard spaces, that is complete geodesic metric spaces with nonpositive curvature in the sense of Alexandrov (see e.g.\ \cite{BridsonHaefliger1999}), covering in particular Hilbert spaces, $\mathbb{R}$-trees, as well as Hadamard manifolds (i.e.\ complete simply connected Riemannian manifolds of nonpositive sectional curvature). This generality is particularly beneficial since stochastic optimization over nonlinear spaces such as manifolds plays a key role in modern machine learning (see e.g.\ \cite{ZhangSra2016}). However, Hadamard spaces are not limited to such settings and beyond that also cover practically relevant spaces without immediate differentiable structure, such as the Billera-Holmes-Vogtmann tree space \cite{BilleraHolmesVogtmann2001} prominently used in phylogenetics, and its recent variation for networks \cite{MoultonSpillner2025}. In all cases, our results not only provide rates at a level of generality not considered previously, but also yield the strong convergence of these methods without local compactness assumptions, which for our abstract notion of regularity is qualitatively novel, even disregarding the quantitative aspects.
	
The first method we discuss is a stochastic variant of the proximal point method for minimizing the mean of a randomized convex function satisfying a Lipschitz condition as studied e.g.\ in \cite{Bacak2018,OhtaPalfia2015,Pischke2026b} (see also \cite{Bacak2014b}). Here we extend the strong convergence results provided in \cite{OhtaPalfia2015} to more general regularity assumptions on the function, and quantitatively outfit them with explicit rates. The second is a randomized variant of the central Krasnoselskii-Mann iteration for solving stochastic common fixed point problems, similar in nature to recent work by Combettes and Madariaga \cite{CombettesMadariaga2025} and recently studied over proper Hadamard spaces in \cite{NeriPischkePowell2026}. Finally, we consider the recently introduced projected subgradient method of Goodwin, Lewis, L\'opez-Acedo and Nicolae \cite{GoodwinLewisLopezAcedoNicolae2024} and its extension studied in \cite{Pischke2026} for stochastic minimization, which utilizes the so-called Busemann subgradients introduced in \cite{GoodwinLewisLopezAcedoNicolae2024}, a new type of subgradient that is particularly suitable for nonlinear geometric contexts.

\subsection{Future work}

We envisage this paper as the basis for future work, not only as a tool for providing effective convergence guarantees for various other methods in stochastic optimization and approximation, but also for instigating further theoretical developments. On the applied side, future work should in particular be concerned with stochastic iterations that make use of our generalized distances. For example, a unified approach to stochastic optimization via Bregman distances, extending stochastic quasi-Fej\'er monotonicity in the particular form treated in \cite{CombettesMadariaga2025}, has been recently considered in \cite{ZhangMiDuSunWangLiZhou2025}, and we anticipate that several of the examples given there can be suitably adapted to our framework. Also, the present work could provide guiding principles for developing effective convergence guarantees for stochastic algorithms with super relaxations in the style of \cite{CombettesMadariaga2025} also over metric contexts such as Hadamard spaces, making use of geodesic rays. On the theoretical as well as applied side, future work could include extensions of the present results to continuous-time processes, combining the approach of this paper (as well as \cite{NeriPischkePowell2025,NeriPischkePowell2026}) with the recent work \cite{FreundPischke2026} of the first author and Freund on similarly broad quantitative considerations of continuous-time dynamical systems in a deterministic setting. These results would presumably rely on a strategy that, in the style of the present paper, combines measurable selection theory with continuous-time martingale theory, a general quantitative approach to the latter being interesting in its own right. Potential applications of these combined results could then in particular include works such as \cite{BotSchindler2024,BotSchindler2025,LiuLongLiHuang2025,LukeSchnebelStaudiglPeypouquetQu2026,MaulenSotoFadiliAttouch2025,MaulenSotoFadiliAttouchOchs2026} on stochastic differential equations and inclusions and related dynamical systems.

\subsection*{Preliminaries and notation}

Throughout, we fix a probability space $(\Omega,\mathsf{F},\PP)$ together with a filtration $(\mathsf{F}_n)$. We denote the (conditional) expectation over that space by $\EE[\cdot]$ and we denote characteristic functions of measurable sets $A\in\mathsf{F}$ by $\mathbf{1}_A$. Further, if not stated otherwise, $(X,d)$ will denote a separable and complete metric space. If seen as a measure space, we assume that $X$ is endowed with its Borel $\sigma$-algebra $\mathsf{B}(X)$. We denote closed balls relative to the metric by 
\[
\overline{B}_r(a):=\{x\in X\mid d(x,a)\leq r\},
\]
given $r>0$ and $a\in X$. We often call a sequence of $X$-valued random variables an $X$-valued stochastic process. If a stochastic process $(x_n)$ is such that $x_n$ is $\mathsf{F}_n$-measurable for any $n\in\mathbb{N}$, we call it adapted to $(\mathsf{F}_n)$. We write $\ell_+(\mathsf{F}_n)$ for the set of all sequences  $(\xi_n)$ of nonnegative random variables adapted to $(\mathsf{F}_n)$ and $\ell^1_+(\mathsf{F}_n)$ for all such sequences that further satisfy $\sum_{n=0}^\infty \xi_n<\infty$ a.s. Above, and throughout the paper, we generally write $\infty$ for $+\infty$, unless we are explicitly differentiating $+\infty$ from $-\infty$.

\section{Carath\'eodory distance functions and stochastic approximation properties}

At our most abstract, we will be concerned with general functions $\phi:X\times X\to [0,\infty)$, potentially distinct from the metric, to measure distances between points. Naturally, not all such functions will be suitable for acting as distances in stochastic contexts, and we have to place some assumptions on the measurability of $\phi$. The central such assumption utilized throughout will be that $\phi$ is a so-called Carath\'eodory distance:

\begin{definition}[Carath\'eodory distance]\label{ass:main}
A function $\phi:X\times X\to [0,\infty)$ is called a Carath\'eodory distance if $\phi(\cdot,y)$ is continuous for any $y\in X$ and $\phi(x,\cdot)$ is measurable for any $x\in X$.
\end{definition}

Being a Carath\'eodory distance is, at our most general level, the only property we require from a function $\phi:X\times X\to [0,\infty)$ to serve as a replacement of the metric in this paper. In particular, we generally will not assume that such functions are symmetric, nor that they satisfy the triangle inequality or even that they are definite. Whenever we speak of an abstract distance $\phi$ as above in the present paper, it is to be implicitly understood that we require it to be a Carath\'eodory distance. If we do require additional assumptions on a Carath\'eodory distance, this will be explicitly highlighted.

For example, as we ultimately care for metric convergence, we will be concerned with converting from a Carath\'eodory distance $\phi$ back to the underlying metric $d$. If points equal in the sense of the Carath\'eodory distance $\phi$ are also equal in the sense of the metric $d$, we call such a distance consistent. The following uniform quantitative formulation of consistency defines our main assumption in that direction:

\begin{definition}[Uniformly consistent distance]
A Carath\'eodory distance $\phi$ is called uniformly consistent with modulus $\theta:(0,\infty)\to (0,\infty)$ if 
\[
\forall \varepsilon>0\ \forall x,y\in X\left(\phi(x,y)<\theta(\varepsilon)\rightarrow d(x,y)<\varepsilon\right).
\]
\end{definition}

Such a modulus essentially induces a growth condition on the distance $\phi$ in terms of the metric, as we highlight in the following remark:

\begin{remark}\label{rem:consistencyGrowth}
Let $\theta:[0,\infty)\to [0,\infty)$ be such that $\theta(0)=0$ and $\theta(\varepsilon)>0$ for $\varepsilon>0$. If $\theta$ is a modulus of uniform consistency for $\phi$, then $\phi(x,y)\geq \theta(d(x,y))$ for all $x,y\in X$.
\end{remark}

Some key examples of (uniformly consistent) distances are collected in the following example:

\begin{example}\label{ex:distances}
The following functions are Carath\'eodory distances:
\begin{enumerate}
\item Perturbed metric distances, that is $\phi=G\circ d$ where $d$ is the metric of the space and $G:[0,\infty)\to [0,\infty)$ is continuous. In particular, if $G$ is inverse continuous at $0$ with a modulus $g:(0,\infty)\to (0,\infty)$, that is
\[
\forall a\geq 0\ \forall \varepsilon >0 \left( G(a)<g(\varepsilon)\to a<\varepsilon\right),
\]
then $\phi$ is uniformly consistent with modulus $\theta(\varepsilon):=g(\varepsilon)$. A common example of a perturbation function is $G=(\cdot)^2$, which is immediately inverse continuous at $0$ with modulus $g(\varepsilon):=\varepsilon^2$. 
\item Perturbed Bregman distances, that is $\phi=G\circ D_f$ for $G$ continuous as above and
\[
D_f(x,y):= f(x)-f(y)-\langle x-y,\nabla f(y)\rangle
\]
over a reflexive Banach space $(X,\norm{\cdot})$, where $f:X\to\mathbb{R}$ is convex and Fr\'echet differentiable on $X$. Suppose further that $G$ is inverse continuous at $0$ with modulus $g$ as above, and that $f$ is sequentially consistent\footnote{Crucially, if $X$ contains at least two points, a function $f$ as above is sequentially consistent iff it is totally convex on bounded sets iff it is uniformly convex on bounded sets (see Theorem 2.10 in \cite{ButnariuResmerita2006} and see also \cite{ButnariuIusem2000}).} with a modulus of sequential consistency $\rho:(0,\infty)^2\to (0,\infty)$ in the sense of \cite{PischkeKohlenbach2024}, i.e.\
\[
\forall \varepsilon,b>0\ \forall x,y\in\overline{B}_b(0)\left( D_f(x,y)<\rho(\varepsilon,b)\to \norm{x-y}<\varepsilon\right).
\]
Then $\phi$ is uniformly consistent on every ball $\overline{B}_b(0)\subseteq X$ with modulus $\theta(\varepsilon):=g(\rho(\varepsilon,b))$. Indeed, if $x,y\in \overline{B}_b(0)$ are such that $G(D_f(x,y))<g(\rho(\varepsilon,b))$, then $D_f(x,y)<\rho(\varepsilon,b)$ as before and so $\norm{x-y}<\varepsilon$. We refer to \cite{PischkeKohlenbach2024} for related discussions.
\end{enumerate}
\end{example}

For a distance $\phi$ and a non-empty set $S\subseteq X$, we write
\[
\mathrm{dist}^\phi_S(x):=\inf_{s\in S}\phi(s,x)
\]
and we omit the superscript $\phi$ only if $\phi=d$. If a distance is consistent, then we can in particular also convert the associated set-distance functions:

\begin{lemma}\label{lem:consistentDistance}
Let $\phi$ be a Carath\'eodory distance that is uniformly consistent with modulus $\theta$. Further, let $S\subseteq X$ be non-empty. Then for any $\varepsilon>0$ and any $x\in X$:
\[
\mathrm{dist}^\phi_S(x)<\theta(\varepsilon)\to \mathrm{dist}_S(x)<\varepsilon.
\]
\end{lemma}
\begin{proof}
Suppose $\mathrm{dist}^\phi_S(x)<\theta(\varepsilon)$. Thus, there exists an $s\in S$ with $\phi(s,x)<\theta(\varepsilon)$. By the assumption on $\theta$, we get $d(s,x)<\varepsilon$ and so in particular $\mathrm{dist}_S(x)<\varepsilon$.
\end{proof}

We now turn to measurability properties of Carath\'eodory distances. Firstly, this assumption entails the following crucial measurability properties on $\phi$ and $\mathrm{dist}^\phi$:

\begin{lemma}\label{lem:measurable}
If $\phi$ is a Carath\'eodory distance, then
\begin{enumerate}
\item $\phi$ is measurable w.r.t.\ $\mathsf{B}(X)\otimes \mathsf{B}(X)$,
\item $\mathrm{dist}^\phi_S$ is measurable for any non-empty $S\subseteq X$.
\end{enumerate}
\end{lemma}
\begin{proof}
Joint measurability of $\phi$, that is item (1), follows from e.g.\ Lemma 8.2.6 in \cite{AubinFrankowska2009}. The measurability of $\mathrm{dist}^\phi_S$ follows from e.g.\ (the proof of) the first part of Lemma 8.2.11 in \cite{AubinFrankowska2009}, noting in particular that for that first part it suffices to just work over a measurable space.
\end{proof}

As a key step in our convergence proofs and construction of associated rates, we will later crucially rely on the property that given a set $S\subseteq X$, we can measurably select points $s\in S$ so that $\phi(s,x)$ approximates $\mathrm{dist}^\phi_S(x)$, and that with arbitrary degree of precision. A result in that vein appears, for $\phi=d$, in the work of R\"omisch \cite{Roemisch1986} and our results effectively extend his. However, \cite{Roemisch1986} operates under the assumption that the underlying $\sigma$-algebra is complete. Indeed, it is exactly this completeness of the underlying probability space, which occurs as a common assumption whenever results from measurable selection theory are used, that is problematic in our context as we will later require this property for all elements $\mathsf{F}_n$ of an associated filtration, which will generally not be complete. Nevertheless, the above property can be guaranteed without any completeness assumptions of the underlying probability space for the present Carath\'eodory distances $\phi$, as we will now show. 

For this, we rely on a few results from measurable selection theory which we now collect. Given a measurable space $(T,\mathsf{T})$ and a complete separable metric space $X$, a set-valued map $\varphi:T\to 2^X$ is called graph measurable if
\[
\mathrm{gra}(\varphi):=\{(t,x)\in T\times X\mid x\in\varphi(t)\}\in\mathsf{T}\otimes\mathsf{B}(X).
\]
Over complete $\sigma$-finite measure spaces, and if $\varphi$ has non-empty closed images, this is equivalent to the (weak) measurability of $\varphi$ as commonly considered in measurable selection theory but generally, the above presents a weaker notion. We refer to \cite{AliprantisBorder2006,AubinFrankowska2009,CastaingValadier1977} for further background on that area.

We begin with the well-known measurable selection theorem of Aumann \cite{Aumann1969}:

\begin{lemma}[Aumann \cite{Aumann1969}, see also Corollary 18.27 in \cite{AliprantisBorder2006}]\label{lem:Aumann}
Let $(T,\mathsf{T},\mu)$ be a finite measure space and let $X$ be a complete separable metric space. Let $\varphi:T\to 2^X$ be a graph measurable with non-empty values. Then there is a measurable function $x:T\to X$ such that $x(t)\in \varphi(t)$ almost everywhere.
\end{lemma}

Beyond this selection theorem, the other main result we need is a variant of the inverse image theorem in measurable selection theory (see e.g.\ Theorem 8.2.9 in \cite{AubinFrankowska2009}). This result is commonly stated under the assumption that the measure space is complete, an assumption which we, as discussed before, crucially want to avoid. We in the following give a variant which dispenses of that assumption while weakening the conclusion to the graph measurability of the respective function, which will however suffice in the context of Aumann's selection theorem. That argument essentially follows the usual proof of the inverse image theorem, and so harbors no surprises. Nevertheless, we rederive it here for the convenience of the reader.

\begin{lemma}\label{lem:inverseTheoremMeas}
Let $(T,\mathsf{T})$ be a measurable space and let $X,Y$ be complete separable metric spaces. Let $\varphi:T\to 2^X$ and $\psi:T\to 2^Y$ be graph measurable and let $c:T\times X\to Y$ be a Carath\'eodory function. Then $\chi$ defined by
\[
\chi(t):=\{x\in \varphi(t)\mid c(t,x)\in \psi(t)\}
\]
is graph measurable.
\end{lemma}
\begin{proof}
Define $d(t,x):=(t,c(t,x))$ and note that
\[
\mathrm{gra}(\chi)=\mathrm{gra}(\varphi)\cap d^{-1}(\mathrm{gra}(\psi)).
\] 
As $c$ is a Carath\'eodory function, it is $\mathsf{T}\otimes \mathsf{B}(X)$/$\mathsf{B}(Y)$-measurable. In particular, we thus have
\[
d^{-1}(A\times B)=c^{-1}(B)\cap (A\times X)\in\mathsf{T}\otimes\mathsf{B}(X)
\]
for any $A\in\mathsf{T}$ and $B\in\mathsf{B}(Y)$. In particular, we thus have $d^{-1}(C)\in \mathsf{T}\otimes\mathsf{B}(X)$ for any $C\in\mathsf{T}\otimes\mathsf{B}(Y)$ and so $d$ is measurable. In particular, as $\mathrm{gra}(\varphi)$ and $\mathrm{gra}(\psi)$ are measurable, we thus have that $\mathrm{gra}(\chi)$ is measurable as well.
\end{proof}

The following result now is an extension of the canonical result that closed balls in complete separable metric spaces are measurable if their origins and radii are (see e.g.\ Corollary 8.2.13 in \cite{AubinFrankowska2009}). For that, given $r>0$ and $x\in X$, we define
\[
\overline{B}^\phi_r(x):=\{y\in X\mid \phi(y,x)\leq r\}
\]
as the closed ball around $x$ with radius $r$, relative to $\phi$.

\begin{lemma}\label{lem:ballsMeas}
Let $(T,\mathsf{T})$ be a measurable space and let $f:T\to X$, $\rho:T\to [0,\infty)$ be measurable and let $\phi$ be a Carath\'eodory distance. Then $t\mapsto \overline{B}^\phi_{\rho(t)}(f(t))$ is graph measurable.
\end{lemma}
\begin{proof}
Define $c(t,x):=\phi(x,f(t))$. As $\phi$ is a Carath\'eodory distance as described above, we get that also $c$ is a Carath\'eodory function. The result now follows from the previous Lemma \ref{lem:inverseTheoremMeas} by setting $\varphi\equiv X$ and $\psi(t):= [0,\rho(t)]$.
\end{proof}

\begin{lemma}\label{lem:approximations}
Let $\mathcal{F}$ be a sub-$\sigma$-algebra of $\mathsf{F}$. Further, let $\phi$ be a Carath\'eodory distance and let $S\in\mathsf{B}(X)$ be non-empty. Then $\mathrm{dist}^\phi_S$ has $\mathcal{F}$-measurable approximations in the sense that for all $\mathcal{F}$-measurable $X$-valued random variables $x$ and any $\varepsilon>0$, there exists some $X$-valued $\mathcal{F}$-measurable random variable $s$ such that
\[
s\in S\text{ and }\phi(s,x)\leq \mathrm{dist}^\phi_S(x)+\varepsilon\text{ a.s.}
\]
\end{lemma}
\begin{proof}
Let $x$ be an $X$-valued $\mathcal{F}$-measurable random variable and let $\varepsilon>0$ be given. Define
\[
A(\omega):=\{s\in S\mid \phi(s,x(\omega))\leq \mathrm{dist}^\phi_S(x(\omega))+\varepsilon\}=S\cap \overline{B}^\phi_{r(\omega)}(x(\omega))
\]
where $r(\omega):=\mathrm{dist}^\phi_S(x(\omega))+\varepsilon$. As $\mathrm{dist}_\phi$ is measurable, $r(\omega)$ is an $\mathcal{F}$-measurable random variable. By Lemma \ref{lem:ballsMeas}, we get that $\overline{B}^\phi_{r(\omega)}(x(\omega))$ is graph measurable w.r.t.\ $\mathcal{F}$. As we have
\[
\mathrm{gra}A=\mathrm{gra}\left(\overline{B}^\phi_{r}(x)\right)\cap (\Omega\times S)
\]
and since $S$ is Borel, we also get that $A$ is graph measurable w.r.t.\ $\mathcal{F}$. Aumann's measurable selection theorem, that is Lemma \ref{lem:Aumann}, now yields the existence of an $X$-valued $\mathcal{F}$-measurable random variable $s$ such that $s\in S$ and $\phi(s,x)\leq \mathrm{dist}^\phi_S(x)+\varepsilon$ a.s.
\end{proof}

We will actually only rely on the above property being true in mean:

\begin{corollary}\label{cor:approximationsMean}
Let $\mathcal{F}$ be a sub-$\sigma$-algebra of $\mathsf{F}$. Further, let $\phi$ be a Carath\'eodory distance and let $S\in\mathsf{B}(X)$ be non-empty. Then $\mathrm{dist}^\phi_S$ has $\mathcal{F}$-measurable approximations in mean, in the sense that for all $\mathcal{F}$-measurable $X$-valued random variables $x$ and any $\varepsilon>0$, there exists some $X$-valued $\mathcal{F}$-measurable random variable $s$ such that 
\[
s\in S\text{ a.s.\ and }\EE[\phi(s,x)]\leq \EE[\mathrm{dist}^\phi_S(x)]+\varepsilon.
\]
\end{corollary}

\section{Stochastic regularity for abstract problems}\label{sec:regularity}

\subsection{Regularity in mean}

As motivated in the introduction already, we consider an (essentially) arbitrary function $F:X\to [0,\infty]$ and the associated problem of finding an element of 
\[
\mathrm{zer}F:=\{z\in X\mid F(z)=0\}
\]
for a general and abstract problem formulation. Naturally, also not all such problem formulations will be permissible in stochastic contexts. We insist on the following mild assumption throughout:

\begin{definition}[Stochastic problem function]\label{ass:mainProblem}
A function $F:X\to [0,\infty]$ is called a stochastic problem function if $F$ is measurable and $\mathrm{zer}F$ is a closed non-empty set.
\end{definition}

As we will see, a range of stochastic problems can be formulated in this simple manner, satisfying the above condition of being a stochastic problem function (see Section \ref{sec:examples} later). In particular, while the function $F$ and its associated zero problem are themselves deterministic, they also naturally cover stochastic zero problems as illustrated abstractly in the following example:

\begin{example}\label{ex:problemEx}
Let $(T,\mathsf{T},\mu)$ be a $\sigma$-finite measure space and $h:T\times X\to [0,\infty]$ be a Carath\'eodory function. The associated stochastic problem
\[
\mathrm{zer}h:=\{z\in X\mid h(t,z)=0\text{ almost everywhere}\}
\]
can be recognized as an instance of a problem $\mathrm{zer}F$ as above for the stochastic problem function $F(z):=\int h(t,z)\,d\mu(t)$, whenever $\mathrm{zer}h$ is nonempty. In particular, note that $F$ is measurable by the Fubini-Tonelli theorem as $h$ is $\mathsf{T}\otimes \mathsf{B}(X)$-measurable. Further, note that $\mathrm{zer}F$ is closed: If $(z_n)\subseteq \mathrm{zer}F$ with $\lim_{n\to\infty}z_n=z$, then $h(t,z_n)=0$ almost everywhere for all $n\in\mathbb{N}$, say on $T^c_n$ with $T_n$ of measure $0$. Define $T':=\bigcup_{n\in\mathbb{N}}T_n$. Then ${T'}$ still has measure $0$ and for $t\in ({T'})^c$, we have $h(t,z_n)=0$ for all $n\in\mathbb{N}$. As $h$ is continuous in its right argument, we have $h(t,z)=0$. Therefore $h(t,z)=0$ almost everywhere, so that $z\in\mathrm{zer}F$.
\end{example}

We are primarily interested in quantitative stochastic regularity conditions on such problems which allow for the construction of explicit rates of convergence of stochastic processes that satisfy a standard almost-supermartingale condition. An important ``regularity'' assumption that is often imposed in this regard, even though often left implicit, is that the respective problem has a unique solution $\mathrm{zer}F=\{z\}$ quantified by an explicit \emph{modulus of uniqueness} in the following stochastic sense:

\begin{definition}[Stochastic uniqueness in mean]\label{def:stochasticUniqueness}
Let $F$ be a stochastic problem function such that $\mathrm{zer}F=\{z\}$, let $\phi$ be a Carath\'eodory distance and let $D$ be a collection of $X$-valued random variables. A modulus of $\phi$-uniqueness for $F$ in mean w.r.t.\ $D$ is a function $\tau:(0,\infty)\to (0,\infty)$ with
\[
\forall \varepsilon>0\ \forall x\in D\left(\EE[F(x)]<\tau(\varepsilon)\to \EE[\phi(z,x)]<\varepsilon\right).
\]
\end{definition}

The more general regularity notion we introduce below arises as a natural extension of this quantitative notion of uniqueness in mean to \emph{non-unique} problems, obtained by replacing the distance $\phi(z,x)$ to the (unique) solution $z$ with the distance $\mathrm{dist}^\phi_{\mathrm{zer}F}(x)$ to the solution set $\mathrm{zer}F$. 

\begin{definition}[Stochastic regularity in mean]
Let $F$ be a stochastic problem function, let $\phi$ be a Carath\'eodory distance and let $D$ be a collection of $X$-valued random variables. A modulus of $\phi$-regularity for $F$ in mean w.r.t.\ $D$ is a function $\tau:(0,\infty)\to (0,\infty)$ with
\[
\forall \varepsilon>0\ \forall x\in D\left(\EE[F(x)]<\tau(\varepsilon)\to \EE[\mathrm{dist}^\phi_{\mathrm{zer}F}(x)]<\varepsilon\right).
\]
\end{definition}
If $D$ is the class of all $X$-valued random variables, we simply call such a function $\tau$ a modulus of $\phi$-regularity for $F$ in mean. We note that this notion coincides with the previous modulus of uniqueness if $\mathrm{zer}F=\{z\}$. Crucially, such a function induces a growth condition on $F$ in mean, as we collect in the following remark.

\begin{remark}\label{rem:meanGrowth}
Let $F$ be a stochastic problem function, let $\phi$ be a Carath\'eodory distance and let $\tau:[0,\infty)\to[0,\infty)$ be such that $\tau(0)=0$ and $\tau(\varepsilon)>0$ for $\varepsilon>0$. If $\tau$ is a modulus of $\phi$-regularity for $F$ in mean w.r.t.\ $D$, then
\[
\EE[F(x)]\geq \tau(\EE[\mathrm{dist}^\phi_{\mathrm{zer}F}(x)])
\]
for all $x\in D$. Further, if $\tau$ is additionally nondecreasing, these two properties are equivalent. To see this equivalence, simply note that for any $\varepsilon>0$, if $\EE[F(x)]<\tau(\varepsilon)$, then it holds that $\tau(\EE[\mathrm{dist}^\phi_{\mathrm{zer}F}(x)])<\tau(\varepsilon)$ which yields that $\EE[\mathrm{dist}^\phi_{\mathrm{zer}F}(x)]<\varepsilon$, as $\tau$ is nondecreasing.
\end{remark}

Instantiations of such general stochastic moduli of regularity appear in various situations already throughout the literature, as we will survey later in this section. It is the goal of this paper to develop a uniform quantitative theory of these regularity assumptions and in particular to illustrate how they can be used to systematically derive rates of convergence for stochastic algorithms. As discussed in the introduction, related (but slightly different) results for the notion of stochastic uniqueness in mean were obtained by the present authors and Neri in \cite{NeriPischkePowell2025} (with moduli of uniqueness in mean called \emph{moduli of strong uniqueness in expectation} therein), derived from more abstract quantitative results for certain stochastic processes, and these also motivate part of our approach here. However, as also highlighted before, the uniqueness assumption heavily simplifies the problem, so that the present paper in particular relies on substantiative additional theory that complements the work \cite{NeriPischkePowell2025}.

\subsection{Variants of stochastic regularity}

We now discuss abstract ways in which regularity in mean can arise, and how corresponding moduli can be derived and defined in these cases. These abstract results are then used later on to capture well-known regularity notions from the literature as instances of our abstract notion. Our first result in this vein shows how stochastic regularity in mean can be obtained from a pointwise property, for suitable $\tau$.

\begin{definition}[Pointwise regularity]
Let $F$ be a stochastic problem function and let $\phi$ be a Carath\'eodory distance. A modulus of $\phi$-regularity for $F$ is a function $\tau:(0,\infty)\to (0,\infty)$ with
\[
\forall \varepsilon>0\ \forall x\in X\left(F(x)<\tau(\varepsilon)\to \mathrm{dist}^\phi_{\mathrm{zer}F}(x)<\varepsilon\right).
\]
\end{definition}

For $\phi=d$, this property coincides with (a special case\footnote{In \cite{KohlenbachLopezAcedoNicolae2019}, the authors further consider a ``local'' variant of this notion, with the property required only on a ball around a fixed solution $z\in\mathrm{zer}F$ (see Definition 3.1 therein). In this paper, we omit this additional locality condition, as it has implications on the boundedness of random variables that seem to limit the stochastic theory. In any way, most regularity notions commonly encountered in the literature are even of this ``global'' form, as we will later survey in Section \ref{sec:examples}, so that this creates no severe limitations.} of) the deterministic notion of a modulus of regularity as defined in \cite{KohlenbachLopezAcedoNicolae2019}. The case of more general distance functions first appeared in \cite{Pischke2025b}. Similarly to before, we can recognize such a regularity property as a different formulation of a growth condition for $F$. This was already highlighted in \cite{KohlenbachLopezAcedoNicolae2019} (see Remark 3.2 therein), after which the following remark (and in fact already Remark \ref{rem:meanGrowth}) are modelled.

\begin{remark}\label{rem:growth}
Let $F$ be a stochastic problem function. Let $\tau:[0,\infty)\to[0,\infty)$ be such that $\tau(0)=0$ and $\tau(\varepsilon)>0$ for $\varepsilon>0$. If $\tau$ is a modulus of $\phi$-regularity for $F$, then
\[
F(x)\geq \tau(\mathrm{dist}^\phi_{\mathrm{zer}F}(x))
\]
for all $x\in X$. Again, if $\tau$ is additionally nondecreasing, these two properties are equivalent, which can be shown as in Remark \ref{rem:meanGrowth}. 
\end{remark}

While in \cite{KohlenbachLopezAcedoNicolae2019} (and \cite{Pischke2025b}) this modulus is studied only in the context of nonstochastic problems, we can now show that in the case that $\tau$ is nondecreasing and convex and $\mathrm{dist}^\phi_{\mathrm{zer}F}(x)$ integrable for all $x\in D$, any such deterministic modulus is also a modulus of $\phi$-regularity for $F$ in mean w.r.t.\ $D$. In particular, it thus follows that essentially all deterministic regularity notions as studied in \cite{KohlenbachLopezAcedoNicolae2019}, which are rather numerous as we will later discuss, immediately entail a stochastic regularity notion with the same modulus under mild conditions that are in most cases satisfied. 

\begin{lemma}\label{lem:RegEquivs}
Let $F$ be a stochastic problem function, let $\phi$ be a Carath\'eodory distance, let $\tau:[0,\infty)\to [0,\infty)$ be convex and nondecreasing with $\tau(0)=0$ and $\tau(\varepsilon)>0$ for all $\varepsilon>0$, and let $D$ be a collection of random variables such that $\mathrm{dist}^\phi_{\mathrm{zer}F}(x)$ is integrable for all $x\in D$. Then, if $\tau$ is a modulus of $\phi$-regularity for $F$, it also a modulus of $\phi$-regularity for $F$ in mean w.r.t.\ $D$.
\end{lemma}
\begin{proof}
Let $\tau$ be a modulus of $\phi$-regularity for $F$. Using Remark \ref{rem:growth}, we in particular have $\tau(\mathrm{dist}^\phi_{\mathrm{zer}F}(x))\leq F(x)$ pointwise everywhere for all $x\in D$. Using that $\tau$ is convex, Jensen's inequality yields
\[
\tau(\EE[\mathrm{dist}^\phi_{\mathrm{zer}F}(x)])\leq \EE[\tau(\mathrm{dist}^\phi_{\mathrm{zer}F}(x))]\leq \EE[F(x)]
\]
for all $x\in D$. As in Remark \ref{rem:meanGrowth}, since $\tau$ is nondecreasing, this yields that $\tau$ is a modulus of $\phi$-regularity for $F$ in mean w.r.t.\ $D$.
\end{proof}

As this relation between pointwise and stochastic regularity has a crucial impact on the range of the stochastic theory laid out in this paper, a natural question is when such \emph{convex} moduli of regularity can be obtained. Interestingly, if $\tau$ satisfies a certain mild growth condition, then even in the case that $\tau$ is not convex, we can construct an alternative modulus of regularity which is convex:

\begin{remark}
Suppose that $\tau:[0,\infty)\to [0,\infty)$ is strictly increasing with $\tau(0)=0$. If $\tau$ satisfies $\liminf_{x\to\infty}\tau(x)/x>0$, then its convex envelope
\[
\check{\tau}(x):=\sup\{f(x)\mid f\leq \tau\text{ is convex}\}
\]
is also strictly increasing. While this result seems folklore, we are not aware of a reference and so provide a proof below. However, first note that in such a case, by virtue of Remark \ref{rem:growth}, the convex envelope $\check{\tau}$ is also a modulus of $\phi$-regularity for $F$ for a given stochastic problem function $F$ and a Carath\'eodory distance $\phi$, provided $\tau$ was one, as we in particular have
\[
F(x)\geq \tau(\mathrm{dist}^\phi_{\mathrm{zer}F}(x))\geq \check{\tau}(\mathrm{dist}^\phi_{\mathrm{zer}F}(x))
\]
for all $x\in X$. Now, to see the above result, note that since $\liminf_{x\to\infty}\tau(x)/x>0$, we have $\tau(x)/x>c>0$ for all $x\geq x_0>0$ for some $c$ and $x_0$. Take $x_1\in (0,x_0)$ and define $m:=\min\{c,\tau(x_1)/(x_0-x_1)\}$. The function $f(x):=m(x-x_1)$ is clearly convex and further satisfies $f\leq \tau$. In particular, we therefore have $0<f(x)\leq \check{\tau}(x)$ for all $x\in (x_1,\infty)$ and as $x_1$ can be chosen arbitrarily close to $0$, we have $\check{\tau}(x)>0$ for all $x\in (0,\infty)$. This implies that $\check{\tau}$ is strictly increasing as for $0<x<y$, using convexity of $\check{\tau}$ and that $\check{\tau}(0)=0$, we have
\[
\check{\tau}(x)=\check{\tau}(\tfrac{x}{y}y)\leq \tfrac{x}{y}\check{\tau}(y)<\check{\tau}(y),
\]
using that $\check{\tau}(y)>0$.
\end{remark}

We return to the setting of the previous Example \ref{ex:problemEx} and now consider problems of the form $F(z):=\int h(t,z)\,d\mu(t)$ where $h:T\times X\to [0,\infty]$ is a Carath\'eodory function over a $\sigma$-finite measure space $(T,\mathsf{T},\mu)$. In that context, we can generally guarantee the existence of a stochastic modulus of regularity already under a relatively broad probabilistic condition, stating that $h$ has a pointwise modulus of regularity with positive probability. This property is readily checked in concrete scenarios, as we discuss briefly in Section \ref{sec:examples}, and in particular bears a resemblance to probabilistic regularity notions recently investigated by Asi and Duchi \cite{AsiDuchi2019} (see Sections 4.1 and 4.2 therein) or Combettes and Madariaga \cite{CombettesMadariaga2025} (see eq.\ (5.7) therein).

\begin{lemma}\label{lem-reg-combettes}
Let $\phi$ be a Carath\'eodory distance. Let $(T,\mathsf{T},\mu)$ be a $\sigma$-finite measure space and let $h:T\times X\to [0,\infty]$ be a Carath\'eodory function such that $\mathrm{zer}h\neq\emptyset$. Set $F(z):=\int h(t,z)\,d\mu(t)$. Suppose that $\tau,\sigma:[0,\infty)\to [0,\infty)$ are nondecreasing functions with $\tau(0),\sigma(0)=0$ and $\tau(\varepsilon),\sigma(\varepsilon)>0$ for $\varepsilon>0$. If
\[
\mu\left(\left\{t\in T\mid h(t,x)\geq \tau(\mathrm{dist}^\phi_{\mathrm{zer}F}(x))\right\}\right)\geq \sigma(\mathrm{dist}^\phi_{\mathrm{zer}F}(x))
\]
for all $x\in X$, then $(\sigma\cdot\tau)(\varepsilon):=\sigma(\varepsilon)\tau(\varepsilon)$ is a modulus of $\phi$-regularity for $F$. In particular, whenever $\tau$ and $\sigma$ are convex and $\mathrm{dist}^\phi_{\mathrm{zer}F}(x)$ is integrable for all $x\in D$, then $\sigma\cdot\tau$ is a modulus of $\phi$-regularity for $F$ in mean w.r.t.\ $D$.
\end{lemma}
\begin{proof}
Given $x\in X$, write $S_x$ for the set $\{t\in T\mid h(t,x)\geq \tau(\mathrm{dist}^\phi_{\mathrm{zer}F}(x))\}$. Then we have
\[
F(x)\geq \int_{S_x}h(t,x)\, d\mu(t)\geq \int_{S_x}\tau(\mathrm{dist}^\phi_{\mathrm{zer}F}(x))\, d\mu(t)\geq \mu(S_x)\cdot \tau(\mathrm{dist}^\phi_{\mathrm{zer}F}(x))\geq (\sigma\cdot \tau)(\mathrm{dist}^\phi_{\mathrm{zer}F}(x))
\]
which yields the first part using Remark \ref{rem:growth}, as $\sigma\cdot\tau$ must also be nondecreasing. The second part then follows from Lemma \ref{lem:RegEquivs} using the standard fact that the product of two convex, nondecreasing, nonnegative functions is convex.
\end{proof}

\begin{remark}
Let $\phi$ be a Carath\'eodory distance and let $(T,\mathsf{T},\mu)$ be a probability space. In the special case of $\sigma$ defined as $\sigma(0):=0$ and $\sigma(\varepsilon):=p$ for $\varepsilon>0$, given a $p\in (0,1]$, Lemma \ref{lem-reg-combettes} represents the following stochastic variant of Lemma \ref{lem:RegEquivs}: If, for any $x\in X$, we have $h(t,x)\geq \tau(\mathrm{dist}^\phi_{\mathrm{zer}F}(x))$ with probability $p$ w.r.t.\ $\mu$, where $\tau$ is a suitable (in particular convex) function, then $p\cdot \tau$ is a modulus of $\phi$-regularity for $F$ in mean w.r.t.\ $D$.
\end{remark}

\subsection{Examples from practice}\label{sec:examples}

There are numerous concrete instantiations of the abstract notions of regularity presented above. As already highlighted, by virtue of Lemma \ref{lem:RegEquivs}, essentially all of the examples studied in \cite{KohlenbachLopezAcedoNicolae2019} immediately lift to the stochastic setting. Among others, these encompass:

\begin{itemize}
\item \emph{Fixed point problems}, formalized via $F(x):=d(Tx,x)$ for some measurable mapping $T:X\to X$ such that $\mathrm{Fix}T\neq\emptyset$ is closed. Here, explicit pointwise moduli of regularity can be constructed when, e.g., $T$ is a quasi-contraction, a continuous orbital contraction, a retraction onto a subset of $X$ or the composition of reflected resolvents for convex semi-algebraic sets in the sense of \cite{BorweinLiTam2017}. Explicit constructions of such pointwise moduli are given in Example 3.6 in \cite{KohlenbachLopezAcedoNicolae2019} (with quasi-contractivity being a simple modification of the construction given therein). Moreover, these moduli of regularity are all linear, and hence in particular convex, and so lift to regularity in mean by Lemma \ref{lem:RegEquivs}.\smallskip
\item \emph{Minimization problems}, formalized via $F(x):=f(x)-\min f$ for some measurable function $f:X\to (-\infty,+\infty]$ such that $\mathrm{argmin}f\neq\emptyset$ is closed. Here, explicit moduli of regularity can be constructed when, e.g., $f$ has a $\tau$-global weak sharp minimum for some strictly increasing $\tau:[0,\infty)\to [0,\infty)$ with $\tau(0)=0$, i.e.\
\[
f(x)- \min f\geq\tau(\mathrm{dist}_{\mathrm{argmin}f}(x))\text{ for all }x\in X.
\]
Concretely, $\tau$ is then immediately a pointwise modulus of regularity. This notion was introduced in \cite{KohlenbachLopezAcedoNicolae2019} (see Example 3.7 therein), extending the well-known notion of weak sharp minima \cite{BurkeFerris1993} (see also \cite{BurkeDeng2002,Ferris1991,LiMordukhovichWangYao2011}). Inverses of this property, that is increasing functions $\tilde{\tau}$ such that
\[
\tilde{\tau}(f(x)- \min f)\geq \mathrm{dist}_{\mathrm{argmin}f}(x)\text{ for all }x\in X,
\]
are also known as error bounds \cite{BolteNguyenPeypouquetSuter2017,DrusvyatskiyLewis2018}. In particular, weak sharp minima as defined in \cite{BurkeFerris1993} arise from the above by considering $\tau(\varepsilon)=k\varepsilon$ for $k>0$, so that this immediately induces regularity in mean by Lemma \ref{lem:RegEquivs}. However, by that lemma, the above pointwise property of course induces a regularity property in mean also for more general convex moduli $\tau$, which in particular further encompasses polynomial growth conditions (see e.g.\ \cite{Shapiro1994} among many others), that is
\[
c(f(x)- \min f)\geq (\mathrm{dist}_{\mathrm{argmin}f}(x))^\theta\text{ for all }x\in X,
\]
for some $c>0$ and $\theta \geq 1$. The quadratic growth condition $\theta=2$ in particular is closely related (and under suitable condition equivalent) to the well-known Polyak-{\L}ojasiewicz, or more generally Kurdyka-{\L}ojasiewicz conditions, where we refer e.g.\ to \cite{BolteDaniilidisLeyMazet2010,BolteNguyenPeypouquetSuter2017,Zhang2020}. A particular situation where a function has weak sharp minima, and which moreover guarantees uniqueness of the solution, is e.g.\ when $f$ is uniformly (and hence in particular strongly) quasiconvex.\smallskip
\item \emph{Set-valued inclusion problems}, formalized via $F(x):=\mathrm{dist}(O_Y,A(x))$ for $A:X\to 2^Y$ where $X$, $Y$ are two given metric spaces and $O_Y\in Y$ is a designated point such that $A^{-1}(O_Y):=\{x\in X\mid O_Y\in A(x)\}\neq\emptyset$ is closed. Here, explicit moduli of regularity can be constructed when e.g. $A$ is $\tau$-global metrically subregular for some strictly increasing $\tau:[0,\infty)\to [0,\infty)$ with $\tau(0)=0$, i.e.\
\[
\mathrm{dist}_{A(x)}(O_Y)\geq \tau(\mathrm{dist}_{A^{-1}(O_Y)}(x))\text{ for all }x\in X.
\]
Concretely, $\tau$ is then immediately a pointwise modulus of regularity. A local variant of this generalized notion of metric subregularity was studied recently in \cite{LiMordukhovichZhu2026}, extending the well-known notion of (local) metric subregularity \cite{Leventhal2009,DontchevRockafellar2009} (a global variant of which was studied, in the context of stochastic iterations, in \cite{HermerLukeSturm2019}). In particular, this usual notion of metric subregularity arises from the above by considering $\tau(\varepsilon)=\varepsilon/k$ for $k>0$, which immediately implies regularity in mean by Lemma \ref{lem:RegEquivs}. However, as before this is not limited to that case but holds more generally whenever $\tau$ is convex, in particular encompassing polynomial growth conditions similar to the above, that is
\[
c\mathrm{dist}_{A(x)}(O_Y)\geq (\mathrm{dist}_{A^{-1}(O_Y)}(x))^\theta\text{ for all }x\in X,
\]
for some $c>0$ and $\theta \geq 1$, and we refer again to \cite{BolteDaniilidisLeyMazet2010,BolteNguyenPeypouquetSuter2017,Zhang2020} for discussions of such conditions for subgradients and other set-valued operators as well as their relation to Polyak-{\L}ojasiewicz or Kurdyka-{\L}ojasiewicz conditions. Another case where such a regularity conditions for set-valued inclusions arise naturally, and moreover guarantee uniqueness of the solution, is when, over Banach spaces, the operator is $\tau$-uniformly accretive for a strictly increasing (convex) $\tau$ as above, covering in particular strongly accretive operators (see Example 3.9 in \cite{KohlenbachLopezAcedoNicolae2019}). This moreover applies to other notions of monotonicity for set-valued operators, in particular uniformly and strongly monotone vector fields over Hadamard manifolds (see e.g.\ \cite{LiLopezMartinMarquez2009}) or Hadamard spaces (see \cite{ChaipunyaKohsakaKumam2021}).
\end{itemize}

These examples already encompass many of the standard regularity assumptions from deterministic and stochastic optimization. In the setting of stochastic optimization, some further problem formulations and associated regularity notions, which are of a genuinely probabilistic nature, feature prominently, and we now illustrate how some of these examples fit into our general notion:

\begin{itemize}
\item Consider the generic (sometimes called online) stochastic minimization problem
\[
\mbox{find some minimizer of }\underline{f}(x):=\int f(e,x)\,d\mu(e)
\]
for some suitable\footnote{A very common assumption on $f$ in particular is that $f$ is a normal convex integrand, that is $f$ is $\mathsf{E}\otimes\mathsf{B}(X)$-measurable and $f(e,\cdot)$ is proper, lower-semicontinuous and convex, which we discuss in our applications later.} function $f:E\times X\to (-\infty,+\infty]$, over some suitable probability space $(E,\mathsf{E},\mu)$. Assume that a minimizer $z\in\mathrm{argmin}\underline{f}$ exists. Instead of requiring a direct growth condition on $\underline{f}$, such as quadratic growth considered e.g.\ in \cite{Shapiro1994}, or the more general conditions surveyed above, which in particular leads to associated expected growth conditions (recall Lemma \ref{lem:RegEquivs}), one can assume a more pointwise and probabilistic condition. Assume for this that the above $z$ is actually a minimizer almost everywhere, that is 
\[
f(e,z)=\inf_{x\in X}f(e,x)\text{ for $\mu$-almost every $e$.}
\]
Problems satisfying this condition are called \emph{interpolation problems} in \cite{AsiChadhaChengDuchi2020} (or \emph{easy problems} in \cite{AsiDuchi2019}, where it is however required that the above condition holds for all solutions $z$; see also e.g.\ \cite{SchmidtLeRoux2013} for similar such conditions). We can then consider the following probabilistic growth condition
\[
\mu\left(\left\{e\in E\mid f(e,x)-f(e,z)\geq \tau(\mathrm{dist}_{\mathrm{argmin}\underline{f}}(x))\right\}\right)\geq p\text{ for all }x\in X,
\]
for some $p\in (0,1]$ and a convex strictly increasing function $\tau:[0,\infty)\to [0,\infty)$ with $\tau(0)=0$, i.e.\ that $f$ satisfies a growth condition induced by $\tau$ with non-zero probability. By Lemma \ref{lem-reg-combettes} (setting $h(e,x):=f(e,x)-f(e,z)$), we thereby in particular get that $p\cdot\tau$ is a modulus of regularity for $F(x):=\underline{f}(x)-\min\underline{f}$ in mean.

This in particular generalizes examples surveyed in \cite{AsiDuchi2019} (see in particular Sections 4.1 and 4.2 therein), where the special cases of $\tau(\varepsilon):=\lambda\varepsilon$ or $\tau(\varepsilon):=\lambda\varepsilon^2$ are considered, which are reasonably easy to check in some situations as discussed in \cite{AsiDuchi2019}, in particular including (see Section 4.3 in \cite{AsiDuchi2019}) overdetermined linear systems, data interpolation problems, or convex feasibility problems for suitable sets (such as halfspaces \cite{AsiDuchi2019}, or more generally convex semi-algebraic sets \cite{BorweinLiTam2017}, as discussed also above).\smallskip
\item Consider the generic stochastic common fixed point problem
\[
\mbox{find some point $z$ such that }T_kz=z \ \PP\mbox{-a.s.}
\]
for some suitable\footnote{A common assumption might be that each $T_k$ is continuous (e.g.\ nonexpansive), and that the function $(k,x)\mapsto T_kx$ is $\mathsf{K}\otimes\mathsf{B}(X)$/$\mathsf{B}(X)$-measurable, as we discuss in our applications later.} family of mappings $(T_k)_{k\in K}$ over some measurable space $(K,\mathsf{K})$ and a $K$-valued random variable $k:\Omega\to K$ over a probability space $(\Omega,\mathsf{F},\PP)$. Assuming that a solution $z\in \mathrm{Fix}T:=\{z\in X\mid T_kz=z \ \PP\mbox{-a.s.}\}$ exists, we can consider the probabilistic condition
\[
\PP\left(\left\{\omega\in \Omega\mid d^2(T_{k(\omega)}x,x)\geq \tau(\mathrm{dist}^{2}_{\mathrm{Fix}T}(x))\right\}\right)\geq p\text{ for all }x\in X,
\]
for some $p\in (0,1]$ and a convex strictly increasing function $\tau:[0,\infty)\to [0,\infty)$ with $\tau(0)=0$, similar to before. Again, by Lemma \ref{lem-reg-combettes} (setting $h(\omega,x):=d^2(T_kx,x)$), we thereby in particular get that $p\cdot\tau$ is a modulus of regularity for $F(x):=\EE[d^2(T_kx,x)]$ in mean. For a linear function $\tau(\varepsilon):=\varepsilon/v$, given some $v\in [1,\infty)$, the above is a probabilistic variant of the common assumption of pointwise \emph{linear regularity} of $(T_k)_{k\in K}$ in the sense of
\[
\mathrm{dist}^2_{\mathrm{Fix}T}(x)\leq v\EE[d^2(T_kx,x)]\text{ for all }x\in X,
\]
as e.g.\ recently considered in the work of Combettes and Madariaga (cf.\ condition (5.10) in \cite{CombettesMadariaga2025}, and also Remark 5.6 of that paper which discusses related literature in which variants of this regularity notion appear), which by Lemma \ref{lem:RegEquivs} of course also immediately induces a resulting regularity modulus for $F$ in mean. A simple example in which linear regularity is achieved is for finite families $T_1,\dots,T_N$ of nonexpansive mappings where $k:\Omega\to \{1,\dots,N\}$ is some random variable with $0<p_i:=\PP(E_i)$ for $E_i:=\{k=i\}$ for all $i=1,\dots,N$. Then $\mathrm{Fix}T=\bigcap_{i=1}^N \mathrm{Fix}T_i$
and linear regularity follows, for example, from the piecewise property
\[
\forall x\in X\left(\mathrm{dist}_{\mathrm{zer}F}(x)\leq ad(T_ix,x)\text{ for some $i\in \{1,\dots,N\}$}\right),
\]
which by Lemma \ref{lem-reg-combettes} yields linear regularity with $v:=a^2/\mu$ for $\mu:=\min_{i\in \{1,\dots,N\}}p_i$.
\end{itemize}

\section{Rates for monotone stochastic processes under regularity conditions}\label{sec:rates}

We now utilise moduli of regularity in mean in order to derive explicit rates of convergence for stochastic methods that solve problems of the form $\mathrm{zer}F$. Here, we will focus on methods that are suitably monotone.

\subsection{Stochastic quasi-Fej\'er monotonicity}

Our central notion is motivated by the property of stochastic quasi-Fej\'er monotonicity:

\begin{definition}[Stochastic quasi-Fej\'er monotonicity]
\label{def:fejer}
Let $(\mathsf{F}_n)$ be a filtration and let $(x_n)$ be an $X$-valued stochastic process adapted to $(\mathsf{F}_n)$. Then $(x_n)$ is called stochastically $\phi$-quasi-Fej\'er monotone w.r.t.\ $S\subseteq X$ and $(\mathsf{F}_n)$ if 
\[
\EE[\phi(z,x_{n+1})\mid\mathsf{F}_n]\leq (1+\zeta_n)\phi(z,x_n)+\xi_n\text{ a.s.}
\]
for all $n\in\mathbb{N}$ and all $z\in S$, where $(\zeta_n),(\xi_n)\in\ell^1_+(\mathsf{F}_n)$.
\end{definition}

Already in the deterministic context, quasi-Fej\'er monotonicity is one of the most fundamental concepts in the modern study of numerical algorithms (see e.g.\ \cite{Combettes2001,Combettes2009}), and this notion retains its relevance in stochastic contexts, where it is typically also referred to simply as a ``supermartingale property'', and indeed the convergence theory of supermartingales fundamentally underlies this notion (in particular the Robbins-Siegmund theorem \cite{RobbinsSiegmund1971}). In Euclidean contexts, this stochastic notion appears already in the pioneering works of Ermol'ev \cite{Ermolev1969,Ermolev1971,ErmolevTuniev1973} together with a wide theory. In the general context of Hilbert spaces, central convergence results (almost surely and in mean) are then presented and streamlined in the seminal work of Combettes and Pesquet \cite{CombettesPesquet2015,CombettesPesquet2019}, which were extended to the metric context of nonlinear Hadamard spaces in the recent work \cite{Pischke2026} (see also \cite{NeriPischkePowell2025,NeriPischkePowell2026} for different metric considerations). 

As it appears above, this general notion of stochastic quasi-Fej\'er monotonicity was already investigated in the preceding work \cite{NeriPischkePowell2025}. While this notion proved suitable for constructing rates of convergence under uniqueness conditions, it turns out that the generality gained by the abstract regularity notions considered here, which in particular allow for non-unique problems, requires us to consider a strengthened version of the above quasi-Fej\'er monotonicity property, where we allow $z$ to be an $S$-valued $\mathsf{F}_n$-measurable random variable:

\begin{definition}[Strong stochastic quasi-Fej\'er monotonicity]
\label{def:fejerStrong}
Let $\phi$ be a Carath\'eodory distance, let $(\mathsf{F}_n)$ be a filtration and let $(x_n)$ be an $X$-valued stochastic process adapted to $(\mathsf{F}_n)$. Then $(x_n)$ is called strongly stochastically $\phi$-quasi-Fej\'er monotone w.r.t.\ $S\subseteq X$ and $(\mathsf{F}_n)$ if 
\[
\EE[\phi(z,x_{n+1})\mid\mathsf{F}_n]\leq (1+\zeta_n)\phi(z,x_n)+\xi_n\text{ a.s.}
\]
for all $n\in\mathbb{N}$ and all $S$-valued $\mathsf{F}_n$-measurable random variables $z$ with $z\in S$ a.s.\ such that $\EE[\phi(z,x_n)]<\infty$, where $(\zeta_n),(\xi_n)\in\ell^1_+(\mathsf{F}_n)$.
\end{definition}

As we will see in Section \ref{sec:applications} below, many standard methods satisfy our strong stochastic quasi-Fej\'er monotonicity property outright, virtue of the fact that they naturally satisfy an even stronger pointwise inequality. However, we can further justify the reach of our strong stochastic quasi-Fej\'er property by showing that it immediately arises from the standard property under a relaxed variant of the triangle inequality for $\phi$, which is satisfied in many cases. In this context, we in particular generalise an argument set out in \cite{CombettesMadariaga2025} to establish a similar fact, tailored to a specific iteration at hand, in Hilbert spaces and for $\phi(x,y):=\norm{x-y}^2$ (contained in their Proposition 2.4 and Theorem 3.2 (iii)).

\begin{definition}[Weak quasi-triangle inequality]
A Carath\'eodory distance $\phi$ satisfies the weak quasi-triangle inequality if there is a concave and nondecreasing function $H:[0,\infty)\to [0,\infty)$ such that
\[
\phi(x,y)\leq H(\phi(x,o)+\phi(y,o))
\]
for any $x,y,o\in X$.
\end{definition}

Related notions are e.g.\ studied in \cite{Grasmair2010} in the context of generalized distance functions. We here restrict ourselves to the canonical examples of $p$-th orders of the metric, similar to Example 3.1 in \cite{Grasmair2010}, and certain examples of Bregman distances.

\begin{example}
The following Carath\'eodory distances satisfy the weak quasi-triangle inequality:
\begin{enumerate}
\item In the case that $\phi=d^q$ for some $q\geq 1$, it follows by the (discrete) Jensen's inequality that
\[
d^q(x,y)\leq 2^{q-1}(d^q(x,o)+d^q(y,o))
\]
for all $x,y,o\in X$. Hence, $\phi=d^q$ satisfies the weak quasi-triangle inequality with function $H(a):= 2^{q-1}a$.
\item In the case that $\phi=D_f$ over a reflexive Banach space $(X,\norm{\cdot})$, where $f:X\to\mathbb{R}$ is convex and Fr\'echet differentiable on $X$, assume further that there are nondecreasing functions $\theta,\Theta:(0,\infty)\to (0,\infty)$ such that $D_f(x,y)\geq\theta(b)\norm{x-y}^2$ and $\norm{\nabla f(x)-\nabla f(y)}\leq\Theta(b)\norm{x-y}$ for any $b>0$ and $x,y\in \overline{B}_b(0)$. Such functions are (essentially) considered over Euclidean spaces $X=\mathbb{R}^d$ in \cite{BauschkeLewis2000}, where it is shown that they exist whenever $f$ is \emph{very strictly convex} (as introduced in \cite{BauschkeLewis2000}), that is if $f$ is additionally twice continuously differentiable and $\nabla^2 f(x)$ is positive definite for any $x\in X$. Related discussions for the infinite dimensional case can be found in \cite{PintoPischke2026}. In particular, by Lemma 2.5 in \cite{PintoPischke2026}, the latter condition on the gradient in particular implies that $D_f(x,y)\leq\Theta(b)\norm{x-y}^2$ for any $b>0$ and $x,y\in \overline{B}_b(0)$. As such, for $b>0$ and $x,y,o\in \overline{B}_b(0)$, we ultimately have
\begin{align*}
D_f(x,y)&\leq \Theta(b)\norm{x-y}^2\\
&\leq \frac{2\Theta(b)}{\theta(b)}\left(\theta(b)\norm{x-o}^2+\theta(b)\norm{y-o}^2\right)\\
&\leq \frac{2\Theta(b)}{\theta(b)}(D_f(x,o)+D_f(y,o)),
\end{align*}
using in particular also item (1) above. Hence, in such a case, $\phi=D_f$ satisfies the weak quasi-triangle inequality on every ball $\overline{B}_b(0)\subseteq X$ with function $H(a):= (2\Theta(b)/\theta(b))a$.
\end{enumerate}
\end{example}

We now give our central result on the relation between stochastic quasi-Fej\'er monotone and strongly stochastic quasi-Fej\'er monotone processes.

\begin{proposition}\label{prop:fejerUpgrade}
Let $X$ be a separable and complete metric space and let $S\subseteq X$ be non-empty and closed. Let $\phi$ be a Carath\'eodory distance. Further, let $(x_n)$ be an $X$-valued stochastic process adapted to $(\mathsf{F}_n)$ which is stochastically $\phi$-quasi-Fej\'er monotone w.r.t.\ $S$ and $(\mathsf{F}_n)$. Suppose that $\EE[\phi(o,x_n)]<\infty$ for all $n\in\mathbb{N}$, where $o\in S$ is fixed. If $\phi$ is such that $\phi(x,x)=0$ for all $x\in X$, and $\phi$ satisfies the weak quasi-triangle inequality with a concave and nondecreasing function $H$, then $(x_n)$ is strongly stochastically $\phi$-quasi-Fej\'er monotone w.r.t.\ $S$ and $(\mathsf{F}_n)$.
\end{proposition}

\begin{proof}
We start by showing that the strong stochastic quasi-Fej\'er property holds for $S$-valued $\mathsf{F}_n$-simple functions $z$, that is random variables $z$ such that there are $z_0,\dots,z_k\in S$ and disjoint $A_0,\dots,A_k\in \mathsf{F}_n$ with $\bigcup_{i=0}^k A_i=\Omega$ and $z(\omega)=z_i$ if, and only if, $\omega\in A_i$. To see this, we observe for such $z$ that
\begin{align*}
\EE[\phi(z,x_{n+1})\mid\mathsf{F}_n]&=\EE\left[\sum_{i=0}^k\phi(z_i,x_{n+1})\mathbf{1}_{A_i}\mid\mathsf{F}_n\right]\\
&=\sum_{i=0}^k\EE[\phi(z_i,x_{n+1})\mid\mathsf{F}_n] \mathbf{1}_{A_i}\\
&\leq (1+\zeta_n)\sum_{i=0}^k \phi(z_i,x_n) \mathbf{1}_{A_i}+\xi_n\\
&=(1+\zeta_n)\phi(z,x_n)+\xi_n,
\end{align*}
where for the second equality, we used that $A_i\in \mathsf{F}_n$. Next, we show that for any $S$-valued $\mathsf{F}_n$-measurable $z$ such that $\EE[\phi(z,x_n)]<\infty$, there exists a sequence $(z_k)$ of $S$-valued $\mathsf{F}_n$-simple functions that converge to $z$ a.s.\ and satisfy $\sup_{k\in\NN}\phi(z_k,o)\leq \phi(z,o)+1$ a.s. To this end, write $\psi(x):=\phi(x,o)$ and let $(p_k)$ be a countable dense subset of $S$ where $p_0$ is chosen to satisfy $\psi(p_0)\leq \inf_{y\in S}\psi(y)+1$, and define for $k\in\mathbb{N}$ and $y\in S$ the set $I_{k,y}\subset\mathbb{N}$ by
\[
I_{k,y}:=\{i\in \{0,\dots,k\}\mid \psi(p_i)\leq \psi(y)+1\},
\]
noting that $0\in I_{k,y}$ for all $k$ and $y$. Now, fixing such a $S$-valued $\mathsf{F}_n$-measurable $z$ with $\EE[\phi(z,x_n)]<\infty$, for each $k\in\mathbb{N}$, define $(A_i^{k,z})$ for $i\in \{0,\dots,k\}$ by
\begin{align*}
A^{k,z}_0&:=\left\{\omega\in\Omega\mid d(z(\omega),p_0)=\min_{j\in I_{k,z(\omega)}} d(z(\omega),p_j)\right\},\\
A^{k,z}_i&:=\left\{\omega\in\Omega\mid i\in I_{k,z(\omega)}\text{ and } \min_{j\in I_{i-1,z(\omega)}}d(z(\omega),p_j)>d(z(\omega),p_i)=\min_{j\in I_{k,z(\omega)}} d(z(\omega),p_j)\right\},
\end{align*}
noting that some of these sets might be empty. It is easy to check that $A_i^{k,z}\in \mathsf{F}_n$ when $z$ is $\mathsf{F}_n$-measurable, using also that $\psi$ is continuous, and moreover we clearly have $\bigcup_{i=0}^k A_i^{k,z}=\Omega$ where this is a union of disjoint sets. Therefore for $k\in\NN$, defining $z_k(\omega):=p_i$ if, and only if, $\omega\in A^{k,z}_i$, we have that $z_k$ is an $S$-valued $\mathsf{F}_n$-simple function. We see by definition that for any $\omega\in\Omega$ we have $\psi(z_k(\omega))\leq \psi(z(\omega))+1$, and moreover
\[
z_k(\omega)=p_{i_k} \text{ for the least $i_k\in I_{k,z(\omega)}$ such that $d(z(\omega),p_{i_k})=\min_{j\in I_{k,z(\omega)}}d(z(\omega),p_j)$},
\]
and since $(p_k)$ is dense in $S$ and $\psi$ is continuous, we have $z_k(\omega)\to z(\omega)$. We now finish the proof by combining the first two steps. Fixing an $S$-valued $\mathsf{F}_n$-measurable $z$ with $\EE[\phi(z,x_n)]<\infty$, and an approximating sequence $(z_k)$ as above, it follows by the first part that
\[
\EE[\phi(z_k,x_{n+1})\mid\mathsf{F}_n]\leq (1+\zeta_n)\phi(z_k,x_n)+\xi_n \ \ \text{a.s.}\label{eq:fejer}\tag{$\ast$}
\]
for any $n,k\in \mathbb{N}$. Using the weak quasi-triangle inequality for $\phi$, we get 
\[
\phi(z,o)\leq H(\phi(z,x_n)+\phi(o,x_n)),
\]
which, using concavity of $H$, yields
\[
\EE[\phi(z,o)]\leq H(\EE[\phi(z,x_n)]+\EE[\phi(o,x_n)])<\infty.
\]
Using the weak quasi-triangle inequality again, together with $\phi(x,x)=0$ for all $x\in X$, we get 
\[
\phi(x_{n+1},o)\leq H(\phi(x_{n+1},x_{n+1})+\phi(o,x_{n+1}))=H(\phi(o,x_{n+1}))
\]
and using the quasi-triangle inequality a third time, as well as that $H$ is nondecreasing, we have 
\begin{align*}
\phi(z_k,x_{n+1})&\leq H(\phi(z_k,o)+\phi(x_{n+1},o))\\
&\leq H(\phi(z,o)+\phi(x_{n+1},o)+1)\\
&\leq H(\phi(z,o)+H(\phi(o,x_{n+1}))+1)=:Y \ \ \text{a.s.}
\end{align*}
So, using the concavity of $H$ twice, as well as that $H$ is nondecreasing, we get
\[
\EE[Y]\leq H(\EE[\phi(z,o)]+H(\EE[\phi(o,x_{n+1})])+1)<\infty
\]
from the above, together with our integrability assumptions. Since $z_k\to z$ and thus $\phi(z_k,x_{n+1})\to \phi(z,x_{n+1})$ a.s.\ by left continuity of $\phi$, by the conditional dominated convergence theorem we have 
\[
\EE[\phi(z_k,x_{n+1})\mid\mathsf{F}_n]\to \EE[\phi(z,x_{n+1})\mid\mathsf{F}_n] \ \ \text{a.s.}
\]
and so taking limits in \eqref{eq:fejer}, using also that $\phi(z_k,x_n)\to \phi(z,x_n)$ a.s., the result is obtained.
\end{proof}

\begin{remark}
The above Proposition \ref{prop:fejerUpgrade} features the assumption that $\EE[\phi(o,x_n)]<\infty$ for all $n\in\mathbb{N}$ and some fixed $o\in S$. This is immediately guaranteed in essentially all cases via the stochastic $\phi$-quasi-Fej\'er monotonicity w.r.t.\ $S$, provided that the initial value $x_0$ of the process satisfies $\EE[\phi(o,x_0)]<\infty$ itself.
\end{remark}

\subsection{Approximation properties and effective rates of convergence}

In the context of a stochastic regularity condition as considered in the previous section, we can now guarantee the convergence of strongly stochastic quasi-Fej\'er monotone processes under a very mild asymptotic approximation property, which takes the following (quantitative) form:

\begin{definition}[$\liminf$-property in mean]
Let $F$ be a stochastic problem function. An $X$-valued stochastic process $(x_n)$ has the $\liminf$-property in mean w.r.t.\ $F$ if $\liminf_{n\to\infty}\EE[F(x_n)]=0$. A function $\varphi:(0,\infty)\times\mathbb{N}\to (0,\infty)$ witnessing this property quantitatively in the sense that
\[
\forall \varepsilon>0\ \forall N\in\mathbb{N}\ \exists n\in [N;\varphi(\varepsilon,N)]\left( \EE[F(x_n)]<\varepsilon\right)
\]
is called a $\liminf$-bound in mean for $(x_n)$ w.r.t.\ $F$.
\end{definition}

Such a $\liminf$-bound in mean can then be combined with a modulus of regularity to give a general construction for a rate of convergence, both in mean and almost surely, for the respective stochastic process. This construction, which will occupy us for the most part of the rest of this section, further depends on some minor data, quantitatively witnessing the characteristic properties of the associated correction terms in the quasi-Fej\'er monotonicity condition. 

Concretely, recall that (strong) stochastic $\phi$-quasi-Fej\'er monotonicity features two stochastic correction terms broadening the supermartingale condition, a multiplicative term $(1+\zeta_n)$ and an additive term $\xi_n$, with $(\zeta_n),(\xi_n)\in\ell^1_+(\mathsf{F}_n)$. In the following, we will slightly upgrade and simultaneously quantitatively resolve these integrability properties as follows. 

\begin{definition}
\begin{enumerate}
\item Let $(\xi_n)\in\ell^1_+(\mathsf{F}_n)$ be such that additionally $\sum_{n=0}^\infty \EE[\xi_n]<\infty$. We say that $\chi:(0,\infty)\to\mathbb{N}$ is a rate of convergence for $\sum_{n=0}^\infty \EE[\xi_n]<\infty$ if it holds that
\[
\forall \varepsilon>0\left( \sum_{n=\chi(\varepsilon)}^\infty \EE[\xi_n]<\varepsilon\right).
\]
\item Let $(\zeta_n)\in\ell^1_+(\mathsf{F}_n)$, that is in other words $\prod_{n=0}^\infty (1+\zeta_n)<\infty$ a.s. We say that $K>0$ is a uniform almost-sure bound for $\prod_{n=0}^\infty (1+\zeta_n)<\infty$ a.s.\ if it holds that
\[
\prod_{n=0}^\infty (1+\zeta_n)<K\text{ a.s.}
\]
\end{enumerate}
\end{definition}

While both requirements are actually qualitative strengthenings of the properties $(\zeta_n),(\xi_n)\in\ell^1_+(\mathsf{F}_n)$, they allows for a much smoother development of the associated quantitative results, and at the same time are practically speaking only very mild restrictions, as many algorithms actually confine to a quasi-Fej\'er monotonicity property where $(\zeta_n)$ is a sequence of reals, and in many cases even is constantly $0$, and where $(\xi_n)$ is summable in mean. With these minor quantitative moduli in place, we now are in the position to give our main abstract quantitative convergence theorem, formulated for consistent distances:

\begin{theorem}\label{thm:main}
Let $F:X\to [0,\infty]$ be a stochastic problem function. Further, let $\phi:X\times X\to [0,\infty)$ be a Carath\'eodory distance and assume that $\phi$ is consistent with a modulus $\theta:[0,\infty)\to [0,\infty)$ which is nondecreasing and convex with $\theta(0)=0$ and $\theta(\varepsilon)>0$ for $\varepsilon>0$. Let $(\mathsf{F}_n)$ be a filtration and let $(x_n)$ be an $X$-valued stochastic process adapted to $(\mathsf{F}_n)$ such that:
\begin{enumerate}
\item $(x_n)$ is strongly stochastically $\phi$-quasi-Fej\'er monotone w.r.t.\ $\mathrm{zer}F$ and $(\mathsf{F}_n)$ and error sequences $(\zeta_n),(\xi_n)\in\ell^1_+(\mathsf{F}_n)$, where $K>0$ is a uniform almost-sure bound for $\prod_{n=0}^\infty (1+\zeta_n)<\infty$ and $\chi:(0,\infty)\to\mathbb{N}$ is a rate of convergence for $\sum_{n=0}^\infty \EE[\xi_n]<\infty$.
\item $(x_n)$ has the $\liminf$-property in mean w.r.t.\ $F$ with a $\liminf$-bound $\varphi:(0,\infty)\times\mathbb{N}\to (0,\infty)$.
\end{enumerate}
Lastly, let $\tau:(0,\infty)\to (0,\infty)$ be a modulus of $\phi$-regularity for $F$ in mean w.r.t.\ $D$, where $D$ is a collection of $X$-valued random variables with $(x_n)\subseteq D$. Then there is a $\mathrm{zer}F$-valued random variable $x$ such that $d(x_n,x)\to 0$ in mean and a.s., with rates
\[
\forall\varepsilon>0\ \forall n\geq \rho(\theta(\varepsilon/2))\left(\EE[d(x_n,x)]<\varepsilon\right)
\]
as well as
\[
\forall \lambda,\varepsilon>0\left(\PP(\exists n\geq \rho(\lambda\theta(\varepsilon/2)) (d(x_n,x)\geq\varepsilon))< \lambda\right)
\]
where $\rho(\varepsilon):=\varphi\left(\tau\left(\varepsilon/3K\right),\chi\left(\varepsilon/3K\right)\right)$.
\end{theorem}
\begin{proof}
Fix $\varepsilon>0$ and let $\delta:=\theta(\varepsilon/2) K^{-1}$ and $N:=\chi(\delta/3)$. By the liminf property, we have
\[
\EE[F(x_n)]<\tau\left(\delta/3\right)
\]
for some $n\in [N;\varphi(\tau(\delta/3),N)]=[N,\rho(\theta(\varepsilon/2))]$. Using that $(x_n)\subseteq D$ and that $\tau$ is a modulus of regularity, we get
\[
\EE[\mathrm{dist}^\phi_{\mathrm{zer}F}(x_{n})]< \delta/3.
\]
Using Corollary \ref{cor:approximationsMean}, let $z$ be an $X$-valued $\mathsf{F}_n$-measurable random variable such that $z\in \mathrm{zer}F$ a.s.\ and
\[
\EE[\phi(z,x_{n})]\leq \EE[\mathrm{dist}^\phi_{\mathrm{zer}F}(x_{n})]+\delta/3<2\delta/3.
\]
Now consider the stochastic process $(U_k)_{k\geq n}$ defined by
\[
U_k:=\frac{\phi(z,x_k)}{y_{k-1}}+\EE\left[\sum_{i=k}^\infty \frac{\xi_i}{y_i}\mid\mathsf{F}_k\right] \text{ where } y_j:=\prod_{i=0}^j (1+\zeta_i).
\]
Since $x_k$ and $z$ are $\mathsf{F}_k$-measurable (the latter since $z$ is already $\mathsf{F}_n$-measurable), by Lemma \ref{lem:measurable}, (1) we have that $\phi(z,x_k)$ and thus $U_k$ is $\mathsf{F}_k$-measurable for all $k\geq n$. Using that $(x_n)$ is strongly stochastically $\phi$-quasi-Fej\'er monotone w.r.t.\ $\mathrm{zer}F$ and $(\mathsf{F}_n)$ and that $z$ is $\mathsf{F}_k$-measurable for all $k\geq n$ and $\EE[\phi(z,x_{n})]<\infty$, it follows by induction that $\EE[\phi(z,x_{k})]<\infty$ for all $k\geq n$. Therefore, the strong stochastic $\phi$-quasi-Fej\'er monotonicity now implies that $(U_k)_{k\geq n}$ is a supermartingale w.r.t.\ $(\mathsf{F}_k)_{k\geq n}$. Concretely, using basic properties of conditional expectations:
\begin{align*}
\EE[U_{k+1}\mid\mathsf{F}_k]&=\EE\left[\frac{\phi(z,x_{k+1})}{y_{k}}\mid \mathsf{F}_k\right]+\EE\left[\EE\left[\sum_{i={k+1}}^\infty \frac{\xi_i}{y_i}\mid\mathsf{F}_{k+1}\right]\mid\mathsf{F}_k\right]\\
&=\frac{\EE\left[\phi(z,x_{k+1})\mid \mathsf{F}_k\right]}{y_k}+\EE\left[\sum_{i={k+1}}^\infty \frac{\xi_i}{y_i}\mid\mathsf{F}_k\right]\\
&\leq \frac{(1+\zeta_k)\phi(z,x_k)}{y_{k}}+\frac{\xi_k}{y_{k}}+\EE\left[\sum_{i={k+1}}^\infty \frac{\xi_i}{y_i}\mid\mathsf{F}_{k}\right]\\
&=\frac{\phi(z,x_{k})}{y_{k-1}}+\EE\left[\sum_{i={k}}^\infty \frac{\xi_i}{y_i}\mid\mathsf{F}_{k}\right]=U_k.
\end{align*}
Further, for any $k\geq n$, we have
\[
U_k= \frac{\phi(z,x_k)}{y_{k-1}}+\EE\left[\sum_{i=k}^\infty \frac{\xi_i}{y_i}\mid\mathsf{F}_k\right]\leq \phi(z,x_k)+\EE\left[\sum_{i=k}^\infty \xi_i\mid\mathsf{F}_k\right]
\]
using $y_j\geq 1$ for all $j\in\NN$. Therefore, taking expectations, we have for any $k\geq n$ that
\[
\EE[U_{k}]\leq \EE[U_n]\leq  \EE[\phi(z,x_{n})]+\sum_{i=n}^\infty\EE[\xi_i]<2\delta/3+\delta/3=\delta,
\]
where we use that $n\geq N=\chi(\delta/3)$. As
\[
\frac{\phi(z,x_k)}{K}\leq \frac{\phi(z,x_k)}{y_{k-1}}\leq U_k,
\]
it follows that $\EE[\phi(z,x_k)K^{-1}]<\delta=\theta(\varepsilon/2) K^{-1}$ and we thus also have $\EE[\phi(z,x_k)]<\theta(\varepsilon/2)$ for all $k\geq n$. As in Remark \ref{rem:consistencyGrowth}, we have $\phi(x,y)\geq \theta(d(x,y))$ for all $x,y\in X$, so that since $\theta$ is convex, we get
\[
\theta(\EE[d(z,x_k)])\leq \EE[\theta(d(z,x_k))] \leq \EE[\phi(z,x_k)]<\theta(\varepsilon/2)
\]
by Jensen's inequality. As $\theta$ is nondecreasing, we get $\EE[d(z,x_k)]<\varepsilon/2$ for all $k\geq n$, and so we have that $(x_n)$ is Cauchy in mean, with rate $\rho(\theta(\varepsilon/2))$. Suppose now $\delta:=\lambda\theta(\varepsilon/2) K^{-1}$ instead. Then as above we get some $n\leq \rho(\lambda\theta(\varepsilon/2))$ with $\EE[U_{k}]<\delta$ for all $k\geq n$, so that by Ville's inequality we have
\begin{align*}
\PP(\exists k\geq \rho(\lambda\theta(\varepsilon/2))(\phi(z,x_k)\geq a))&= \PP(\exists k\geq \rho(\lambda\theta(\varepsilon/2))(\phi(z,x_k)/K\geq a/K))\\
&\leq \PP(\exists k\geq \rho(\lambda\theta(\varepsilon/2))(U_k\geq a/K))\\
&\leq \frac{\EE[U_{\rho(\lambda\theta(\varepsilon/2))}]}{a K^{-1}}<\frac{\delta}{a K^{-1}}
\end{align*}
for any $a>0$. Setting $a:=\theta(\varepsilon/2)$, we observe that
\begin{align*}
\PP(\exists k,l\geq\rho(\lambda\theta(\varepsilon/2))(d(x_k,x_l)\geq \varepsilon)&\leq \PP(\exists k,l\geq \rho(\lambda\theta(\varepsilon/2))(d(z,x_k)+d(z,x_l)\geq \varepsilon)) \\
&\leq \PP(\exists k\geq \rho(\lambda\theta(\varepsilon/2))(d(z,x_k)\geq \varepsilon/2))\\
&\leq \PP(\exists k\geq \rho(\lambda\theta(\varepsilon/2))(\phi(z,x_k)\geq \theta(\varepsilon/2)))\\
&<\frac{\delta}{\theta(\varepsilon/2)K^{-1}}=\lambda.
\end{align*}
This shows that $(x_n)$ is Cauchy a.s.\ with rate $\rho(\lambda\theta(\varepsilon/2))$, and by inspecting the above calculations we also see that $\mathrm{dist}_{\mathrm{zer}F}(x_n)\to 0$ a.s., with the same rate. In particular, we thus have that $(x_n)$ almost surely converges to some measurable $x$,\footnote{The existence of a measurable almost sure limit $x$ crucially relies on $X$ being a separable and complete metric space. Concretely, as already used throughout the proof, the set on which $(x_n)$ is Cauchy is measurable since the space is separable (see e.g.\ Lemma 6.1 in \cite{Klenke2020}), having measure one. Modifying the sequence to a fixed value on the complement of this set, we get that the resulting sequence converges pointwise everywhere, using that the space is complete. The function $x$ can be taken as the resulting limit, which is measurable (see e.g.\ Lemma 4.29 in \cite{AliprantisBorder2006}).} and it follows immediately that this is with the same rate. Lastly, we also have
\[
\mathrm{dist}_{\mathrm{zer}F}(x)\leq \mathrm{dist}_{\mathrm{zer}F}(x_n)+d(x_n,x)
\]
so that $\mathrm{dist}_{\mathrm{zer}F}(x)=0$ a.s. As $\mathrm{zer}F$ is closed, we have $x\in \mathrm{zer}F$ a.s. It follows rather immediately that $(x_n)$ also converges to this limit in mean with rate $\rho(\theta(\varepsilon/2))$. Concretely, for $n\geq \rho(\theta(\varepsilon/2))$, we have
\[
\EE[d(x_n,x)]=\EE[\liminf_{m\to\infty} d(x_n,x_{n+m})]\leq \liminf_{m\to\infty} \EE[d(x_n,x_{n+m})]<\varepsilon
\]
using Fatou's lemma.
\end{proof}

\begin{remark}
While formulated for Caratheodory distances $\phi$ above, it follows directly from the proof that Theorem \ref{thm:main} holds whenever
\begin{enumerate}
\item $\phi$ is measurable w.r.t.\ $\mathsf{B}(X)\otimes \mathsf{B}(X)$,
\item $\mathrm{dist}^\phi_{\mathrm{zer}F}$ is measurable for any non-empty closed $S\subseteq X$,
\item for any $n\in\mathbb{N}$, $\mathrm{dist}^\phi_{\mathrm{zer}F}$ has $\mathsf{F}_n$-measurable approximations in mean w.r.t.\ $D$, that is for all $\mathsf{F}_n$-measurable $x\in D$ and any $\varepsilon>0$, there exists some $X$-valued $\mathsf{F}_n$-measurable random variable $z$ such that $z\in \mathrm{zer}F$ a.s.\ and $\EE[\phi(z,x)]\leq \EE[\mathrm{dist}^\phi_{\mathrm{zer}F}
(x)]+\varepsilon$.
\end{enumerate}
\end{remark}

\begin{remark}
It follows immediately by inspection of the proof that the convexity assumption for the consistency modulus $\theta$ featuring in Theorem \ref{thm:main} is not necessary to establish the almost-sure convergence together with the corresponding rate. In particular, towards establishing convergence in mean, this assumption may hence be bypassed by, instead, assuming the uniform integrability of the sequence. This can in particular be made quantitative by assuming corresponding so-called moduli of uniform integrability for the sequence as studied in \cite{PischkePowell2024} (see also \cite{NeriPischkePowell2026}), but we do not discuss this here any further.
\end{remark}

If we only care for the distance to the solution set, the assumption of strong stochastic quasi-Fej\'er monotonicity can be weakened to the ``normal'' variant:

\begin{theorem}\label{thm:mainWeak}
Let $F:X\to [0,\infty]$ be a stochastic problem function. Further, let $\phi:X\times X\to [0,\infty)$ be a Carath\'eodory distance. Let $(\mathsf{F}_n)$ be a filtration and let $(x_n)$ be an $X$-valued stochastic process adapted to $(\mathsf{F}_n)$ such that:
\begin{enumerate}
\item $(x_n)$ is stochastically $\phi$-quasi-Fej\'er monotone w.r.t.\ $\mathrm{zer}F$ and $(\mathsf{F}_n)$ and error sequences $(\zeta_n),(\xi_n)\in\ell^1_+(\mathsf{F}_n)$, where $K>0$ is a uniform almost-sure bound for $\prod_{n=0}^\infty (1+\zeta_n)<\infty$ and $\chi:(0,\infty)\to\mathbb{N}$ is a rate of convergence for $\sum_{n=0}^\infty \EE[\xi_n]<\infty$.
\item $(x_n)$ has the $\liminf$-property in mean w.r.t.\ $F$ with a $\liminf$-bound $\varphi:(0,\infty)\times\mathbb{N}\to (0,\infty)$.
\end{enumerate}
Lastly, let $\tau:(0,\infty)\to (0,\infty)$ be a modulus of $\phi$-regularity for $F$ in mean w.r.t.\ $D$, where $D$ is a collection of $X$-valued random variables with $(x_n)\subseteq D$. Then $\mathrm{dist}^\phi_{\mathrm{zer}F}(x_n)\to 0$ in mean and a.s., with rates
\[
\forall\varepsilon>0\ \forall n\geq \rho(\varepsilon)\left(\EE[\mathrm{dist}^\phi_{\mathrm{zer}F}(x_n)]<\varepsilon\right)
\]
as well as
\[
\forall \lambda,\varepsilon>0\left(\PP(\exists n\geq \rho(\lambda\varepsilon) (\mathrm{dist}^\phi_{\mathrm{zer}F}(x_n)\geq\varepsilon))< \lambda\right)
\]
where $\rho$ is as in Theorem \ref{thm:main}. If $\phi$ is uniformly consistent with a modulus $\theta:[0,\infty)\to [0,\infty)$ which is nondecreasing and convex with $\theta(0)=0$ and $\theta(\varepsilon)>0$ for $\varepsilon>0$, then we further have $\mathrm{dist}_{\mathrm{zer}F}(x_n)\to 0$ in mean and a.s., with rates $\rho(\theta(\varepsilon))$ and $\rho(\lambda\theta(\varepsilon))$ respectively.
\end{theorem}
\begin{proof}
Given any $z\in\mathrm{zer}F$, we have $\mathrm{dist}^\phi_{\mathrm{zer}F}(x_{n+1})\leq \phi(z,x_{n+1})$ and hence we get
\[
\EE[\mathrm{dist}^\phi_{\mathrm{zer}F}(x_{n+1})\mid\mathsf{F}_n]\leq \EE[\phi(z,x_{n+1})\mid\mathsf{F}_n]\leq (1+\zeta_n)\phi(z,x_n)+\xi_n
\]
for any such $z\in\mathrm{zer}F$. Taking the infimum over $z$, we get
\[
\EE[\mathrm{dist}^\phi_{\mathrm{zer}F}(x_{n+1})\mid\mathsf{F}_n]\leq (1+\zeta_n)\mathrm{dist}^\phi_{\mathrm{zer}F}(x_{n})+\xi_n.
\]
The remainder of the proof now follows (a simplification of) the arguments for Theorem \ref{thm:main} (now using Lemma \ref{lem:measurable}, (2) to establish that the main supermartingale is adapted to $(\mathsf{F}_n)$) and is omitted.
\end{proof}

It is important to stress that our abstract results can be refined to fit more specific conditions on $(x_n)$ and $F$, and in that way be used to guarantee stronger convergence guarantees. For example, if the associated regularity modulus is linear, and the error sequences are decaying suitably fast, then we can also obtain linear rates of convergence, even in the form of non-asymptotic guarantees. This can be done by utilizing an almost identical strategy as that applied to the case of unique zeros in \cite{NeriPischkePowell2025}.\footnote{See Theorem 5.7 therein. In fact, the present result arises as a direct application of Theorem 3.6 given in \cite{NeriPischkePowell2025}. However, we present the (rather short) argument tailored to the present situation for completeness.} In that case, we will assume that the process $(x_n)$ is stochastically quasi-Fej\'er monotone in a \emph{strict sense}, that is it satisfies the stricter inequality
\[
\EE[\phi(z,x_{n+1})\mid\mathsf{F}_n]\leq (1+\zeta_n)\phi(z,x_n)-\eta_n F(x_n)+\xi_n\text{ a.s.}
\]
for all $n\in\mathbb{N}$ and all $z\in \mathrm{zer}F$, matching more closely the notion of stochastic quasi-Fej\'er monotonicity studied e.g.\ over Hilbert spaces in \cite{CombettesPesquet2015}. Beyond this slightly extended property, we further rely on a folklore quantitative result for real recursive inequalities (see e.g.\ \cite{NemirovskiJuditskyLanShapiro2009} for similar such results):

\begin{lemma}[see e.g.\ Lemma 3.5 in \cite{NeriPischkePowell2025}]\label{lem:Nemirovski}
Suppose that $(x_n)$ is a sequence of nonnegative reals such that for $c>1$, $d\geq 0$ and $r\in \NN\backslash\{0\}$, we have
\[
x_{n+1}\leq \left(1-\frac{c}{n+r}\right)x_n+\frac{d}{(n+r)^2}
\]
for all $n\in\NN$. Then for all $n\in\NN$:
\[
x_n\leq \frac{u}{n+r} \ \text{ for }\ u\geq \max\left\{\frac{d}{c-1},rx_0\right\}.
\]
\end{lemma}

Our result on fast non-asymptotic guarantees is then readily derived:

\begin{theorem}\label{thm:fast}
Let $F:X\to [0,\infty]$ be a stochastic problem function. Further, let $\phi:X\times X\to [0,\infty)$ be a Carath\'eodory distance. Let $(\mathsf{F}_n)$ be a filtration and let $(x_n)$ be an $X$-valued stochastic process adapted to $(\mathsf{F}_n)$ which is strictly stochastically $\phi$-quasi-Fej\'er monotone w.r.t.\ $\mathrm{zer}F$ and $(\mathsf{F}_n)$, i.e.
\[
\EE[\phi(z,x_{n+1})\mid\mathsf{F}_n]\leq (1+\zeta_n)\phi(z,x_n)-\eta_n F(x_n)+\xi_n\text{ a.s.}
\]
for all $n\in\mathbb{N}$ and all $z\in \mathrm{zer}F$, where $(\zeta_n),(\eta_n)$ are sequences of nonnegative reals and $(\xi_n)$ are nonnegative random variables such that
\[
\EE[\xi_n]\leq d/(n+r)^2 \text{ and } \zeta_n+c/(n+r)\leq t\eta_n
\]
for all $n\in\mathbb{N}$ where $c>1$, $d\geq 0$ and $r\in\mathbb{N}\setminus\{0\}$. Further suppose that $K\geq 1$ and $L>0$ are such that $\prod_{i=0}^\infty (1+\zeta_i)<K$ and $L\geq \EE[\mathrm{dist}^\phi_{\mathrm{zer}F}(x_0)]$. Lastly, let $D$ be a collection of $X$-valued random variables with $(x_n)\subseteq D$ and such that $\EE[F(x)]\geq t\EE[\mathrm{dist}^\phi_{\mathrm{zer}F}(x)]$ for all $x\in D$. Then
\[
\EE[\mathrm{dist}^\phi_{\mathrm{zer}F}(x_n)]\leq\frac{u}{n+r}  \text{ for } u\geq \max\left\{\frac{d}{c-1},rL\right\}
\]
as well as 
\[
\PP\left(\exists m\geq n(\mathrm{dist}^\phi_{\mathrm{zer}F}(x_{m})\geq \varepsilon)\right)\leq \frac{1}{\varepsilon}\cdot\frac{K(u+2d)}{n+r}.
\]
\end{theorem}
\begin{proof}
Similarly to Theorem \ref{thm:mainWeak}, we obtain
\[
\EE[\mathrm{dist}^\phi_{\mathrm{zer}F}(x_{n+1})\mid\mathsf{F}_n]\leq (1+\zeta_n)\mathrm{dist}^\phi_{\mathrm{zer}F}(x_{n})-\eta_n F(x_n)+\xi_n.
\]
Integrating the inequality, we get
\begin{align*}
\EE[\mathrm{dist}^\phi_{\mathrm{zer}F}(x_{n+1})]&\leq (1+\zeta_n)\EE[\mathrm{dist}^\phi_{\mathrm{zer}F}(x_{n})]-\eta_n\EE[F(x_n)]+\EE[\xi_n]\\
&\leq (1+\zeta_n-t\eta_n)\EE[\mathrm{dist}^\phi_{\mathrm{zer}F}(x_{n})]+\EE[\xi_n]\\
&\leq \left(1-\frac{c}{n+r}\right)\EE[\mathrm{dist}^\phi_{\mathrm{zer}F}(x_{n})]+\frac{d}{(n+r)^2}.
\end{align*}
Applying Lemma \ref{lem:Nemirovski} yields the rate for $\EE[\mathrm{dist}^\phi_{\mathrm{zer}F}(x_{n})]$. For the second claim, we proceed similar to the proof of Theorem \ref{thm:main}. Concretely, note first that
\[
\EE[\xi_n]\leq\frac{d}{(n+r)^2}\leq \frac{2d}{(n+r)(n+r+1)}=2d\left(\frac{1}{n+r}-\frac{1}{n+r+1}\right)
\]
so that
\[
\sum_{i=n}^\infty\EE[\xi_i]\leq 2d\sum_{i=n}^\infty\left(\frac{1}{i+r}-\frac{1}{i+r+1}\right)=\frac{2d}{n+r}.
\]
We define
\[
U_{n}:=\frac{\mathrm{dist}^\phi_{\mathrm{zer}F}(x_{n})}{y_{n-1}}+\EE\left[\sum_{i=n}^\infty \frac{\xi_i}{y_i}\mid \mathsf{F}_n\right], \text{ where }y_j:=\prod_{i=0}^j (1+\zeta_i)
\]
and, analogously to Theorem \ref{thm:main} (now using Lemma \ref{lem:measurable}, (2) to establish measurability), we can then derive that $(U_n)$ is a supermartingale. Combining the rate for $\EE[\mathrm{dist}^\phi_{\mathrm{zer}F}(x_{n})]$ with the above bounds yields $\EE[U_n]\leq (u+2d)(n+r)$. Using Ville's inequality, this yields 
\[
\PP\left(\exists m\geq n(\mathrm{dist}^\phi_{\mathrm{zer}F}(x_m)\geq \varepsilon)\right)\leq \PP\left(\exists m\geq n(U_m\geq \varepsilon/K)\right)\leq \frac{K}{\varepsilon}\cdot \EE[U_n]\leq \frac{1}{\varepsilon}\cdot\frac{K(u+2d)}{n+r}.\qedhere
\]
\end{proof}

\section{Applications to stochastic algorithms}
\label{sec:applications}

In this final section we apply our results to concrete stochastic algorithms. In each case, we work in the rather general context of geodesic metric spaces with nonpositive curvature. These spaces were introduced by Alexandrov \cite{Aleksandrov1951} and are often called $\CAT$ spaces after Gromov \cite{Gromov1987}. We refer to \cite{AlexanderKapovitchPetrunin2023,BridsonHaefliger1999} for a comprehensive overview of geodesic and $\CAT$ spaces and further refer to \cite{Bacak2014a} for a shorter treatment focused on aspects of convex analysis and optimization.

We introduce background from this class of spaces only as needed in each application. Beyond this, we only need the following few notions. In a metric space $(X,d)$, geodesics are isometries $\gamma:[0,l]\to X$, said to join $\gamma(0)$ and $\gamma(l)$. The space is called (uniquely) geodesic if every two points are joined by a (unique) geodesic. A geodesic metric space $(X,d)$ is called a $\CAT$ space (also called a space of nonpositive curvature in the sense of Alexandrov) if it satisfies 
\[
d^2(\gamma(tl),x)\leq (1-t)d^2(\gamma(0),x)+td^2(\gamma(l),x)-t(1-t)d^2(\gamma(0),\gamma(l))\tag{CN}\label{CN}
\]
for all $x\in X$, $t\in [0,1]$ and all geodesics $\gamma:[0,l]\to X$, (an extension of) the so-called Bruhat-Tits $\mathrm{CN}$-inequality \cite{BruhatTits1972}. Any $\CAT$ space is uniquely geodesic, and a complete $\CAT$ space is called a Hadamard space. In Hadamard spaces, we generally write $(1-\lambda)x\oplus \lambda y$ for the point $\gamma(\lambda d(x,y))$ on the unique geodesic $\gamma:[0,d(x,y)]\to X$ joining $x$ and $y$.

\subsection{Stochastic proximal point methods}

The first method we study will be the classic stochastic proximal point method. In a deterministic context, where the method originates with the work of Rockafellar \cite{Rockafellar1976}, Martinet \cite{Martinet1970} as well as Br\'ezis and Lions \cite{BrezisLions1978}, the proximal point method was extended to Hadamard spaces by Ba\v{c}\'ak \cite{Bacak2013}, establishing weak convergence (which, by G\"uler's seminal work \cite{Gueler1991}, is the most one can hope for already in Hilbert spaces).

The stochastic proximal point method, widely studied over Euclidean and Hilbert spaces (we refer to \cite{AsiDuchi2019,Bertsekas2011,Bertsekas2012,Bianchi2016,NemirovskiJuditskyLanShapiro2009,RyuBoyd}, among many others, for various such discussions), was lifted to the setting of (separable) Hadamard spaces in the work \cite{Bacak2018}, building on preceding work \cite{Bacak2014b} on a splitting proximal point method with random order for finite sums of convex functions over similar spaces (see also \cite{Pischke2025} for a recent related variant for general perturbed strongly monotone vector fields over Hadamard spaces). As it is in particular over such spaces, without additional differential structure, where proximal point methods gain relevance compared to gradient descent (we refer to \cite{Bacak2018} as well as \cite{Bertsekas2011} for further discussions), we focus on these extensions. 

Let $(\Omega,\mathsf{F},\PP)$ and $(E,\mathsf{E},\mu)$ be probability spaces, with $(E,\mathsf{E},\mu)$ complete, and let $X$ now be a separable Hadamard space. In analogy to \cite{Rockafellar1971} (see also \cite{CastaingValadier1977}), let $f:E\times X\to (-\infty,+\infty]$ be a normal convex integrand, i.e.\ $f(e,\cdot)$ is proper, lower-semicontinuous (lsc) and convex\footnote{Given a Hadamard space $X$, recall that a function $f:X\to (-\infty,+\infty]$ is called lsc if $\liminf_{n\to\infty} f(x_n)\geq f(x)$ whenever $x_n\to x$ in $X$, and convex if $f\circ \gamma$ is convex for any geodesic $\gamma$ in $X$.} for all $e\in E$ and $f$ is $\mathsf{E}\otimes\mathsf{B}(X)$-measurable. Defining $\underline{f}(x):=\int f(e,x)\,d\mu(e)$ and assuming that $\underline{f}$ is lsc and proper and that $\mathrm{argmin}\underline{f}\neq\emptyset$, our problem is to
\[
\mbox{find some element of }\mathrm{argmin} \underline{f}.
\]
We capture this problem in our general setup via $F(x):=\underline{f}(x)-\min\underline{f}$ (recall Section \ref{sec:examples}). Note that $\underline{f}(x)$ and hence $F$ are measurable by Fubini's theorem and that $\mathrm{argmin} \underline{f}=\mathrm{zer}F$ is closed since $\underline{f}$ is lsc.

To now introduce the method, define the proximal map of $f$ via
\[
\mathrm{prox}_{\lambda}^f(e,x):=\mathrm{argmin}_{y\in X}\left\{ f(e,y)+\frac{1}{2\lambda}d^2(x,y)\right\},
\]
which is well-defined for all $e\in E$, $x\in X$ and $\lambda>0$ (see e.g.\ \cite{Jost1995} or \cite{Mayer1998}). Further, $\mathrm{prox}_{\lambda}^f(e,\cdot)$ is nonexpansive for any $e\in E$ and $\lambda>0$ (see e.g.\ Lemma 4 in \cite{Jost1995}), and also $\mathrm{prox}_{\lambda}^f(\cdot,x)$ is measurable for any $x\in X$ and $\lambda>0$. Hence, $\mathrm{prox}_{\lambda}^f$ is a Carath\'eodory function and so in particular $\mathsf{E}\otimes\mathsf{B}(X)$-measurable.

The stochastic proximal point method is then given by the iteration
\[
x_{n+1}:=\mathrm{prox}^f_{\lambda_n}(\xi_{n+1},x_n),\tag{\textsf{SPPA}}\label{SPPA}
\]
given a starting point $x_0\in X$ and sequences $(\lambda_n)$ of positive reals as well as $(\xi_{n+1})$ of random variables $\Omega\to E$, for which we assume that 
\[
(\xi_{n+1})\text{ are i.i.d.\ with distribution $\mu$ and }\sum_{n\in\mathbb{N}}\lambda_n=\infty, \sum_{n\in\mathbb{N}}\lambda_n^2<\infty.\tag{\textsf{SPPA}-\textsf{A1}}\label{SPPAparameters}
\]
The work \cite{Bacak2018} in particular relies on a certain weak growth condition on the integrand introduced therein, which is a generalization of many of the common growth conditions from the literature, in particular of Lipschitz continuity of the functional. For simplicity however, we here assume the following (slightly stronger) Lipschitz-type assumption (used throughout the literature on this method in linear spaces, see also e.g.\ \cite{OhtaPalfia2015} for spaces with bounded curvature): Assume there exists a positive function $L\in L^2(E,\mu)$ with $\int L\,d\mu>0$ such that
\[
\vert f(e,x)-f(e,y)\vert\leq L(e)d(x,y)\tag{\textsf{SPPA}-\textsf{A2}}\label{SPPAgrowth}
\]
for all $x,y\in X$ and almost all $e\in E$.\footnote{This assumption in particular implies that $\underline{f}$ is Lipschitz and hence lsc.}

Now, as mentioned above, without any regularity assumptions the deterministic proximal point method in general only converges weakly, which extends to the stochastic proximal point method as only recently established by the first author \cite{Pischke2026b}.

In both cases, there are no effective convergence guarantees in general. These however can be obtained via additional regularity assumptions. The most common assumption used in the literature is that of strong or at least uniform convexity of the function. Next to a very large number of works in linear spaces, such rates for the deterministic case over Hadamard spaces are also discussed in \cite{Bacak2013} (see also \cite{LeusteanSipos2018}). The stochastic case, in particular over Hadamard spaces, is considerably less populated and the main strong convergence result under a regularity assumption in that vein appears, to our knowledge, in \cite{OhtaPalfia2015}. However, no explicit rates are given in that work.

We here present the following general result on the effective convergence of \eqref{SPPA} under a stochastic regularity assumption:\footnote{See also \cite{KohlenbachLopezAcedoNicolae2019} for a result with similar generality for the deterministic proximal point method in Hilbert spaces.}

\begin{theorem}\label{SPPAstrongConv}
Let $(E,\mathsf{E},\mu)$ and $(\Omega,\mathsf{F},\PP)$ be probability spaces, with $(E,\mathsf{E},\mu)$ complete, and let $X$ be a separable Hadamard space. Let $f:E\times X\to (-\infty,+\infty]$ be a normal convex integrand such that $\underline{f}(x):=\int f(e,x)\,d\mu(e)$ is proper and $\mathrm{argmin}\underline{f}\neq\emptyset$. Write $F(x):=\underline{f}(x)-\min\underline{f}$. Let $(x_n)$ be the iteration given by \eqref{SPPA}, and assume \eqref{SPPAparameters} as well as \eqref{SPPAgrowth}. Lastly, let $\tau:(0,\infty)\to (0,\infty)$ be a modulus of regularity for $F$ in mean w.r.t.\ $D$, i.e.
\[
\forall \varepsilon>0\ \forall x\in D\left(\EE[\underline{f}(x)-\min\underline{f}]<\tau(\varepsilon)\to \EE[\mathrm{dist}^2_{\mathrm{argmin}\underline{f}}(x)]<\varepsilon\right),
\]
where $D$ is a collection of $X$-valued random variables with $(x_n)\subseteq D$. Then $(x_n)$ strongly converges to an $\mathrm{argmin}\underline{f}$-valued random variable $x$ a.s.\ and in mean. Moreover, the following rates of convergence apply: Let $z\in\mathrm{argmin}\underline{f}$ and let $b>d(x_0,z)$. Write $\underline{L}:=\int L^2\,d\mu$. Assume that $T > \sum_{n=0}^\infty \lambda_n^2$ and let $\theta:\mathbb{N}\times (0,\infty)\to\mathbb{N}$ as well as $\chi:(0,\infty)\to\mathbb{N}$ be such that $\sum_{n=\chi(\varepsilon)}^\infty \lambda_n^2<\varepsilon$ for all $\varepsilon>0$ and $\sum_{n=k}^{\theta(k,b)}\lambda_n\geq b$ for all $b>0$ and $k\in\mathbb{N}$. Then $\EE[d(x_n,x)]\to 0$ with rate $\rho(\varepsilon^2/4)$ and $d(x_n,x)\to 0$ a.s.\ with rate $\rho(\lambda\varepsilon^2/4)$, where
\[
\rho(\varepsilon):=
\theta\left(\chi\left(\frac{\varepsilon}{24\underline{L}}\right),\frac{b+4\underline{L}T}{\tau\left(\varepsilon/6\right)}\right).
\]
\end{theorem}

The regularity assumption featuring in the above Theorem \ref{SPPAstrongConv} can now be concretely instantiated in various ways, covering in particular notions such as weak sharp minima or error bounds (recall Section \ref{sec:examples}). The following corollary explicitly spells this out for the special case of quadratic growth, which already seems to be a case where the a.s.\ strong convergence of the iteration (even without any quantitative information) outside of locally compact Hadamard spaces is novel to the literature. We also mention that under linear regularity assumptions and suitable conditions on the parameters, our result on fast rates (recall Theorem \ref{thm:fast}) can be used to obtain linear non-asymptotic guarantees.

\begin{corollary}
Let $(E,\mathsf{E},\mu)$ and $(\Omega,\mathsf{F},\PP)$ be probability spaces, with $(E,\mathsf{E},\mu)$ complete, and let $X$ be a separable Hadamard space. Let $f:E\times X\to (-\infty,+\infty]$ be a normal convex integrand such that $\underline{f}(x):=\int f(e,x)\,d\mu(e)$ is proper and $\mathrm{argmin}\underline{f}\neq\emptyset$. Let $(x_n)$ be the iteration given by \eqref{SPPA}, and assume \eqref{SPPAparameters} as well as \eqref{SPPAgrowth}. Assume that $\underline{f}$ satisfies the quadratic growth condition
\[
c(\underline{f}(x)- \min \underline{f})\geq (\mathrm{dist}_{\mathrm{argmin}\underline{f}}(x))^2\text{ for all }x\in X,
\]
for some $c>0$. Then $(x_n)$ strongly converges to an $\mathrm{argmin}\underline{f}$-valued random variable $x$ a.s.\ and in mean. Moreover, the following rates of convergence apply: For $b$, $\underline{L}$, $T$, $\theta$ and $\chi$ as in Theorem \ref{SPPAstrongConv}, we have $\EE[d(x_n,x)]\to 0$ with rate $\rho(\varepsilon^2/4)$ and $d(x_n,x)\to 0$ a.s.\ with rate $\rho(\lambda\varepsilon^2/4)$, where
\[
\rho(\varepsilon):=
\theta\left(\chi\left(\frac{\varepsilon}{24\underline{L}}\right),\frac{c(b+4\underline{L}T)}{\varepsilon/6}\right).
\]
\end{corollary}

This corollary in particular covers the setup of \cite{OhtaPalfia2015} for strongly convex functions, at least in Hadamard spaces, and as such also that of \cite{Bacak2014b,Bacak2018,Pischke2026b}. Concretely, assume that $f(e,\cdot)$ is strongly convex with parameter $\alpha(e)>0$, i.e.\
\[
f(e,(1-t)x\oplus ty)\leq (1-t)f(e,x)+tf(e,y) -t(1-t)\frac{\alpha(e)}{2}d^2(x,y)
\]
for any $x,y\in X$, any $t\in [0,1]$ and any $e\in E$, where additionally $\underline{\alpha}:=\int \alpha\,d\mu>0$. Then $\underline{f}$ satisfies the quadratic growth condition with $c:=\frac{8}{\underline{\alpha}}$. By virtue of \eqref{CN}, a particular (class of) function(s) that satisfies the strong convexity requirement is the square metric $d^2$ with $c:=4$, so that the above result in particular applies to the instantiation of \eqref{SPPA} applied for finding Fr\'echet means (see again \cite{Bacak2014b,Bacak2018,OhtaPalfia2015}).

To prove Theorem \ref{SPPAstrongConv}, we have to establish the strong stochastic quasi-Fej\'er monotonicity of the sequence and derive a corresponding $\liminf$-bound. Both of these rest on the following fundamental property of the proximal map:

\begin{lemma}[see e.g.\ Lemma 2.2.23 in \cite{Bacak2014a}]\label{resLemma}
For any $\lambda>0$, $x,y\in X$ and $e\in E$:
\[
f(e,\mathrm{prox}^f_{\lambda}(e,x))-f(e,y)\leq\frac{1}{2\lambda} d^2(x,y)-\frac{1}{2\lambda}d^2(\mathrm{prox}^f_{\lambda}(e,x),y).
\]
\end{lemma}

The strong stochastic quasi-Fej\'er monotonicity then follows rather immediately. For that, we introduce the following set $\mathsf{F}_n:=\sigma(\xi_{1},\dots,\xi_{n})$ and we abbreviate $\EE[\cdot\mid\mathsf{F}_n]$ by $\EE_n$. Further, we write $\underline{L}:=\int L^2\,d\mu<\infty$.

\begin{lemma}[extending \cite{Bacak2018}]\label{SPPAIneqMain}
Let $n\in\mathbb{N}$. Then for any $X$-valued $\mathsf{F_n}$-measurable random variable $y$:
\[
\EE_n[d^2(x_{n+1},y)]\leq d^2(x_n,y)-2\lambda_n(\underline{f}(x_n)-\underline{f}(y))+4\lambda_n^2\underline{L} \text{ a.s.}
\]
In particular, if $y$ is additionally such that $y\in\mathrm{argmin}\underline{f}$ a.s., then
\[
\EE_n[d^2(x_{n+1},y)]\leq d^2(x_n,y)-2\lambda_n(\underline{f}(x_n)-\min\underline{f})+4\lambda_n^2\underline{L} \text{ a.s.}
\]
\end{lemma}
\begin{proof}
Given $n\in\mathbb{N}$ and $y\in X$, Lemma \ref{resLemma} implies
\[
d^2(x_{n+1},y)\leq d^2(x_n,y)-2\lambda_n[f(\xi_{n+1},x_{n+1})-f(\xi_{n+1},y)]
\]
and so we immediately have
\begin{align*}
\EE_n[d^2(x_{n+1},y)]&\leq d^2(x_n,y)-2\lambda_n\EE_n[f(\xi_{n+1},x_{n+1})-f(\xi_{n+1},y)]\\
&=d^2(x_n,y)-2\lambda_n\EE_n[f(\xi_{n+1},x_n)-f(\xi_{n+1},y)]\\
&\qquad+2\lambda_n\EE_n[f(\xi_{n+1},x_n)-f(\xi_{n+1},x_{n+1})]\\
&=d^2(x_n,y)-2\lambda_n[\underline{f}(x_n)-\underline{f}(y)]\\
&\qquad+2\lambda_n\EE_n[f(\xi_{n+1},x_n)-f(\xi_{n+1},x_{n+1})]
\end{align*}
where the third equality follows by independence of $\xi_{n+1}$ to $x_n$ and $y$ (using that $y$ is $\mathsf{F}_n$-measurable), as well as the fact that $\xi_{n+1}$ has distribution $\mu$. Now, note that Lemma \ref{resLemma} together with \eqref{SPPAgrowth} yields
\begin{align*}
d^2(x_{n+1},x_n)&\leq 2\lambda_n[f(\xi_{n+1},x_n)-f(\xi_{n+1},x_{n+1})]\\
&\leq 2\lambda_nL(\xi_{n+1})d(x_n,x_{n+1})
\end{align*}
so that we have $d(x_{n+1},x_n)\leq 2\lambda_nL(\xi_{n+1})$. Using \eqref{SPPAgrowth}, we further have
\[
f(\xi_{n+1},x_n)-f(\xi_{n+1},x_{n+1})\leq L(\xi_{n+1})d(x_n,x_{n+1})\leq 2\lambda_nL^2(\xi_{n+1}).
\]
In particular, we have $2\lambda_n\EE_n[f(\xi_{n+1},x_n)-f(\xi_{n+1},x_{n+1})]\leq 4\lambda_n^2\underline{L}$, again using the independence of $\xi_{n+1}$ to $\mathsf{F}_n$. Combined, we get
\[
\EE_n[d^2(x_{n+1},y)]\leq d^2(x_n,y)-2\lambda_n(\underline{f}(x_n)-\underline{f}(y))+4\lambda_n^2\underline{L}
\]
as claimed.
\end{proof}

We can now use that result to derive a $\liminf$-bound. For that require two preliminary results which we shall use later on again. The first is a quantitative version of a lemma of Qihou \cite{Qihou2001} (see also Lemma 5.31 in \cite{BauschkeCombettes2017}):

\begin{lemma}[Theorem 3.2 in \cite{NeriPowell2024}]\label{qihou}
Let $(x_n)$, $(\alpha_n)$, $(\beta_n)$ and $(\gamma_n)$ be sequences of nonnegative reals with
\[
x_{n+1}\leq (1+\alpha_n)x_n-\beta_n+\gamma_n
\]
for all $n\in\mathbb{N}$. If $\prod_{i=0}^\infty(1+\alpha_i)<\infty$ and $\sum_{i=0}^\infty\gamma_i<\infty$, then $(x_n)$ converges and $\sum_{i=0}^\infty\beta_i<\infty$. Further, if $K,L,M>0$ satisfy $x_0<K$, $\prod_{i=0}^\infty(1+\alpha_i)<L$ and $\sum_{i=0}^\infty\gamma_i<M$, then $\sum_{i=0}^\infty\beta_i<L(K+M)$.
\end{lemma}

The next result is folklore, but we include the very brief proof for completeness:

\begin{lemma}\label{sumconv}
Suppose that $(u_n)$, $(v_n)$ are sequences of nonnegative reals with $L>0$ such that $\sum_{n=0}^\infty u_nv_n< L$ and $\theta:\mathbb{N}\times(0,\infty)\to \mathbb{N}$ such that $\sum_{n=k}^{\theta(k,b)} u_n\geq b$ for all $b>0$ and $k\in\mathbb{N}$. Then $\liminf_{n\to\infty}v_n=0$ with 
\[
\forall\varepsilon>0\ \forall N\in\mathbb{N}\ \exists n\in [N;\theta(N,L/\varepsilon)](v_n<\varepsilon).
\]
\end{lemma}
\begin{proof}
For arbitrary $\varepsilon>0$ and $N\in\mathbb{N}$, suppose for a contradiction that $v_n\geq\varepsilon$ for all $n\in [N;\theta(N,L/\varepsilon)]$. Then $L\leq\varepsilon\sum_{n=N}^{\theta(N,L/\varepsilon)}u_n\leq \sum_{n=N}^{\theta(N,L/\varepsilon)} u_nv_n\leq \sum_{n=0}^{\infty} u_nv_n < L$, which is a contradiction.
\end{proof}

\begin{lemma}\label{SPPAliminf}
Let $z\in\mathrm{argmin}\underline{f}$ and let $b>d(x_0,z)$. Let $\theta:\mathbb{N}\times (0,\infty)\to\mathbb{N}$ be such that $\sum_{n=k}^{\theta(k,b)}\lambda_n\geq b$ for all $b>0$ and $k\in\mathbb{N}$, and let $T > \sum_{n=0}^\infty \lambda_n^2$. Then we have
\[
\forall \varepsilon>0\ \forall N\in\mathbb{N}\ \exists n\in [N;\varphi(\varepsilon,N)]\left( \EE[F(x_n)]<\varepsilon\right)
\]
where $\varphi(\varepsilon,N):=\theta(N,(b+4\underline{L}T)/\varepsilon)$.
\end{lemma} 
\begin{proof}
Taking expectations in Lemma \ref{SPPAIneqMain}, applied to $z$, and applying Lemma \ref{qihou} yields the bound $\sum_{n=0}^\infty \lambda_n\EE[\underline{f}(x_n)-\min\underline{f}]<b+4\underline{L}T$ and so Lemma \ref{sumconv} yields the claim.
\end{proof}

\begin{proof}[Proof of Theorem \ref{SPPAstrongConv}]
Note that $\chi(\varepsilon/4\underline{L})$ is a rate of convergence for $\sum_{n\in\mathbb{N}}4\lambda_n^2\underline{L}<\infty$ as we have $\sum_{n=\chi(\varepsilon/4\underline{L})}^\infty \lambda_n^2<\frac{\varepsilon}{4\underline{L}}$ and so $\sum_{n=\chi(\varepsilon/4\underline{L})}^\infty 4\lambda_n^2\underline{L}<\varepsilon$ for all $\varepsilon>0$. The result now immediately follows from Theorem \ref{thm:main}, using in particular Lemmas \ref{SPPAIneqMain} and \ref{SPPAliminf}, and using that $d^2$ is uniformly consistent with modulus $\theta(\varepsilon):=\varepsilon^2$ (recall Example \ref{ex:distances}).
\end{proof}

\subsection{Krasnoselskii-Mann schemes for common fixed point problems}

The second method we study is a randomized Krasnoselskii-Mann scheme for solving stochastic common fixed point problems. Going back to \cite{Krasnoselskii1955,Mann1953} in a deterministic context, the Krasnoselskii-Mann scheme is one of the most central fixed point iterations in modern optimization.

Various stochastic versions of that scheme, tailored to different problem settings, have been proposed, notably  (relatively straightforward) modifications by stochastic noise (see e.g.\ \cite{CombettesPesquet2015,CombettesPesquet2019}). We here focus on a variant of the Krasnoselskii-Mann scheme featuring a randomized selection from a class of operators as e.g.\ recently investigated (in a broad context) over Hilbert spaces by Combettes and Madariaga \cite{CombettesMadariaga2025}. The variant we study here has been previously introduced over Hadamard spaces in work of the authors together with Neri \cite{NeriPischkePowell2026}, where in particular almost-sure convergence of that method over proper Hadamard spaces is established (see Theorem 5.12 therein), though this is based on a compactness argument rather than a regularity assumption, leading to a very different proof strategy which does not allow us to produce such effective convergence rates as in the present paper.

Let $(\Omega,\mathsf{F},\PP)$ be an arbitrary probability space, $(K,\mathsf{K})$ some other measurable space, and $X$ be a separable Hadamard space. Let $(T_k)_{k\in K}$ be a family of mappings on $X$, and furthermore assume that 
\[
\begin{cases}\text{each }T_k\text{ is nonexpansive and the mapping }K\times X\to X\\
\text{defined by }(k,x)\mapsto T_kx\text{ is }\mathsf{K}\otimes\mathsf{B}(X)/\mathsf{B}(X)\text{-measurable}.\end{cases}\tag{\textsf{SKM}-\textsf{A1}}\label{SKMparameters1}
\]
Let $k:\Omega\to K$ be a $K$-valued random variable. Our problem is to
\[
\mbox{find some element of }\mathrm{Fix}T:=\{x\in X\mid T_kx=x \ \PP\mbox{-a.s.}\},
\]
assuming that $\mathrm{Fix}T\neq\emptyset$. We capture this problem in our general setup via $F(x):=\EE[d^2(T_kx,x)]$ (recall Section \ref{sec:examples}). Indeed, note that the map $(\omega,x)\mapsto d^2(T_{k(\omega)}x,x)$ is a Carath\'eodory function since $d^2$ is continuous in both arguments and $T_{k(\omega)}$ is nonexpansive. Hence, it is in particular $\mathsf{F}\otimes \mathsf{B}(X)$-measurable, and so $F$ is measurable, again by Fubini's theorem. Further, note that $\mathrm{Fix}T=\mathrm{zer}F$ is closed (see e.g.\ Lemma 5.2 in \cite{NeriPischkePowell2026}).

To solve the above problem, we now consider a randomised Krasnoselskii-Mann scheme
\[
x_{n+1}:=(1-\lambda_n)x_n\oplus \lambda_n T_{k_n}x_n,\tag{\textsf{SKM}}\label{SKM}
\]
given some starting point $x_0\in X$, sequences $(\lambda_n)$ of random variables $\Omega\to (0,1]$ and $(k_n)$ of random variables $\Omega\to X$, such that  
\[
\begin{cases}(k_n)\text{ is i.i.d.\ and distributed as }k\text{ and }(\lambda_n)\text{ is independent itself}\\
\text{as well as independent of }(k_n)\text{ with }\sum_{n\in\mathbb{N}}\EE[\lambda_n(1-\lambda_n)]=\infty.
\end{cases}\tag{\textsf{SKM}-\textsf{A2}}\label{SKMparameters2}
\]
Analogous to the previous section, we now present a general result on the effective convergence of \eqref{SKM} under a stochastic regularity assumption.

\begin{theorem}\label{SKMstrongConv}
Let $(\Omega,\mathsf{F},\PP)$ be a probability space, $(K,\mathsf{K})$ a measurable space, and $X$ a separable Hadamard space. Let $(T_k)_{k\in K}$ be a family of mappings satisfying \eqref{SKMparameters1}, and let $k:\Omega\to K$ be a $K$-valued random variable such that $\mathrm{Fix}T\neq\emptyset$. Write $F(x):=\EE[d^2(T_kx,x)]$. Let $(x_n)$ be the iteration given by \eqref{SKM} and assume additionally \eqref{SKMparameters2}. Lastly, let $\tau:(0,\infty)\to (0,\infty)$ be a modulus of regularity for $F$ in mean w.r.t.\ $D$, i.e.
\[
\forall \varepsilon>0\ \forall x\in D\left(\EE_{\omega\sim \PP}[\EE_{\omega'\sim\PP}[d^2(T_{k(\omega')}x(\omega),x(\omega))]]<\tau(\varepsilon)\to \EE[\mathrm{dist}^2_{\mathrm{Fix}T}(x)]<\varepsilon\right),
\]
where $D$ is a collection of $X$-valued random variables with $(x_n)\subseteq D$.

Then $(x_n)$ strongly converges to a $\mathrm{Fix}T$-valued random variable $x$ a.s.\ and in mean. Moreover, the following rates of convergence apply: Let $z\in \mathrm{Fix}T$ with $b>d(x_0,z)$, and assume that $\theta:\NN\times (0,\infty)\to \NN$ is such that $\sum_{n=k}^{\theta(k,b)}\EE[\lambda_n(1-\lambda_n)]\geq b$ for all $b>0$ and $k\in\mathbb{N}$. Then $\EE[d(x_n,x)]\to 0$ with rate $\rho(\varepsilon^2/4)$ and $d(x_n,x)\to 0$ a.s.\ with rate $\rho(\lambda\varepsilon^2/4)$ where
\[
\rho(\varepsilon):=\theta\left(0,\frac{b}{\tau(\varepsilon/6)}\right).
\]
\end{theorem}

Also the regularity assumption featuring in the above Theorem \ref{SKMstrongConv} can now be concretely instantiated in various ways (recall again Section \ref{sec:examples}). The following corollary explicitly spells this out for the special case where $(T_k)$ is linearly regular.

\begin{corollary}
Let $(\Omega,\mathsf{F},\PP)$ be a probability space, $(K,\mathsf{K})$ a measurable space, and $X$ a separable Hadamard space. Let $(T_k)_{k\in K}$ be a family of mappings satisfying \eqref{SKMparameters1}, and let $k:\Omega\to K$ be a $K$-valued random variable such that $\mathrm{Fix}T\neq\emptyset$. Let $(x_n)$ be the iteration given by \eqref{SKM} and assume additionally \eqref{SKMparameters2}. Assume that $(T_k)_{k\in K}$ is linearly regular, i.e.
\[
\mathrm{dist}^2_{\mathrm{Fix}T}(x)\leq v\EE[d^2(T_kx,x)]\text{ for all }x\in X,
\]
for some $v\geq 1$. Then $(x_n)$ strongly converges to a $\mathrm{Fix}T$-valued random variable $x$ a.s.\ and in mean. Moreover, the following rates of convergence apply: For $b$ and $\theta$ as in Theorem \ref{SKMstrongConv}, we have $\EE[d(x_n,x)]\to 0$ with rate $\rho(\varepsilon^2/4)$ and $d(x_n,x)\to 0$ a.s.\ with rate $\rho(\lambda\varepsilon^2/4)$ where
\[
\rho(\varepsilon):=\theta\left(0,\frac{vb}{\varepsilon/6}\right).
\]
\end{corollary}

If the step sizes further satisfy a growth condition like $\frac{v}{n+r}\leq \EE[\lambda_n(1-\lambda_n)]$, our result on fast rates (recall Theorem \ref{thm:fast}) in particular applies, which yields linear non-asymptotic guarantees of the form $\EE[\mathrm{dist}^2_{\mathrm{zer}F}(x_n)]\leq u/(n+r)$ for a suitable constant $u>0$, and similarly in the almost sure case.

Theorem \ref{SKMstrongConv} is proven using a similar strategy to that of the previous subsection, establishing first the strong stochastic quasi-Fej\'er monotonicity, and then a suitable $\liminf$-bound. These both echo standard calculations associated with Krasnoselskii-Mann schemes, generalised to Hadamard spaces and the stochastic mappings $(T_k)$. Define the filtration $\mathsf{F}_n:=\sigma(x_0,\dots,x_n,k_0,\dots,k_{n-1},\lambda_0,\dots,\lambda_{n-1})$ and write $\EE_n$ for the conditional expectation $\mathbb{E}[\cdot \mid\mathsf{F}_n]$.

\begin{lemma}[extending \cite{NeriPischkePowell2026}]\label{SKMIneqMain}
Let $n\in\mathbb{N}$. Then for any $X$-valued $\mathsf{F_n}$-measurable random variable $y$:
\begin{align*}
\EE_n[d^2(x_{n+1},y)]\leq &d^2(x_n,y)-\EE[\lambda_n(1-\lambda_n)]\EE_n[d^2(T_{k_n}x_n,x_n)]\\
&\qquad+\EE_n[d^2(T_{k_n}y,y)]+d(x_n,y)\EE_n[d(T_{k_n}y,y)]\text{ a.s.}
\end{align*}
In particular, if $y$ is additionally such that $y\in \mathrm{Fix}T$ a.s., we have
\[
\EE_n[d^2(x_{n+1},y)]\leq d^2(x_n,y)-\EE[\lambda_n(1-\lambda_n)] \EE_n[d^2(T_{k_n}x_n,x_n)] \text{ a.s.}
\]
\end{lemma}
\begin{proof}
Given $n\in\NN$ and $y\in X$, by the triangle inequality and the fact that $T_k$ is nonexpansive, we have
\[
d^2(T_{k_n}x_n,y)\leq d^2(x_n,y)+d^2(T_{k_n}y,y)+d(x_n,y)d(T_{k_n}y,y)
\]
Using \eqref{CN} (see also Lemma 5.7 in \cite{NeriPischkePowell2026}), we get
\[
d^2(x_{n+1},y)\leq (1-\lambda_n)d^2(x_n,y)+\lambda_n d^2(T_{k_n}x_n,y)-\lambda_n(1-\lambda_n) d^2(T_{k_n}x_n,x_n).
\]
Combined, we obtain
\begin{align*}
d^2(x_{n+1},y)&\leq d^2(x_n,y)+\lambda_n d^2(T_{k_n}y,y)+\lambda_n d(x_n,y)d(T_{k_n}y,y)-\lambda_n(1-\lambda_n) d^2(T_{k_n}x_n,x_n)\\
&\leq d^2(x_n,y)+d^2(T_{k_n}y,y)+d(x_n,y)d(T_{k_n}y,y)-\lambda_n(1-\lambda_n) d^2(T_{k_n}x_n,x_n).
\end{align*}
If $y$ is now a $\mathsf{F}_n$-measurable random variable, using basic properties of the conditional expectation, and noting that $\lambda_n(1-\lambda_n)$ is independent to both $d^2(T_{k_n}x_n,x_n)$ and $\mathsf{F}_n$, we have
\begin{align*}
\EE_n[d^2(x_{n+1},y)]\leq &d^2(x_n,y)-\EE[\lambda_n(1-\lambda_n)]\EE_n[d^2(T_{k_n}x_n,x_n)]\\
&\qquad+\EE_n[d^2(T_{k_n}y,y)]+d(x_n,y)\EE_n[d(T_{k_n}y,y)]
\end{align*}
which was the first claim. If now $y$ also satisfies $y\in \mathrm{Fix}T$ a.s., since $k_n$ is independent of $\mathsf{F}_n$ with the same distribution as $k$, we have
\[
\EE_n[d^2(T_{k_n}y,y)](\omega)=\EE_{\omega'\sim\PP}[d^2(T_{k(\omega')}y(\omega),y(\omega))]=0
\]
on a set of measure one. Jensen's inequality yields $\EE_n[d(T_{k_n}y,y)]=0$ a.s., from which the second claim follows.
\end{proof}

\begin{lemma}\label{SKMliminf}
Let $z\in \mathrm{Fix}T$ and let $b>d(x_0,z)$. Let $\theta:\NN\times (0,\infty)\to \NN$ be such that $\sum_{n=k}^{\theta(k,b)}\EE[\lambda_n(1-\lambda_n)]\geq b$ for all $b>0$ and $k\in\mathbb{N}$. Then we have
\[
\forall \varepsilon>0\ \forall N\in\mathbb{N}\ \exists n\in [N;\varphi(\varepsilon,N)]\left( \EE[F(x_n)]<\varepsilon\right)
\]
where $\varphi(\varepsilon,N):=\theta(N,b/\varepsilon)$.
\end{lemma}

\begin{proof}
Taking expectations in Lemma \ref{SKMIneqMain}, applied to $z$, and applying Lemma \ref{qihou} yields the bound $\sum_{n=0}^\infty \EE[\lambda_n(1-\lambda_n)]\EE[d^2(T_{k_n}x_n,x_n)]<b$. The result then follows from Lemma \ref{sumconv} together with the observation that
\[
\EE[d^2(T_{k_n}x_n,x_n)]=\EE_{\omega\sim\PP}[\EE_{\omega'\sim \PP}[d^2(T_{k(\omega')}x_n(\omega),x_n(\omega))]]=\EE[F(x_n)]
\]
by Fubini's theorem and the fact that $k_n$ is independent of $x_n$ and is distributed as $k$.
\end{proof}

\begin{proof}[Proof of Theorem \ref{SKMstrongConv}]
The result is now immediate from Theorem \ref{thm:main}, using Lemmas \ref{SKMIneqMain} and \ref{SKMliminf}, and again that $d^2$ is uniformly consistent with modulus $\theta(\varepsilon):=\varepsilon^2$ (recall Example \ref{ex:distances}).
\end{proof}

\subsection{Stochastic Busemann subgradient methods}

With our last method we return to the problem of minimizing the mean of a stochastic convex function, where we now focus on a projected subgradient method recently introduced by Goodwin, Lewis, L\'opez-Acedo and Nicolae \cite{GoodwinLewisLopezAcedoNicolae2024} which employs the novel notion of Busemann subgradients, and in fact focus on an extension of that method recently studied in \cite{Pischke2026} for general stochastic minimization. This novel type of subgradient coincides with the usual notion of subgradients over Euclidean spaces but, by exploiting the boundary cone $CX^\infty$ of a Hadamard space $X$, associated Busemann functions, and other advanced geometric tools from Hadamard spaces, supports a theory that is particularly suitable for nonlinear geometric contexts.

To introduce this method, we require a bit more background on Hadamard spaces $X$: A geodesic ray in $X$ is an isometry $r:[0,\infty)\to X$, said to be issuing from $r(0)$. $X$ has the geodesic extension property if for all $x\neq y\in X$, there is a ray $r:[0,\infty)\to X$ issuing from $x$ such that $r(t)=y$ for some $t>0$.

The so-called boundary of $X$ at infinity $X^\infty$ is the set of all equivalence classes of rays in $X$ under the equivalence relation of being asymptotic (see Definition 8.1 in \cite{BridsonHaefliger1999}). The boundary cone $CX^\infty$ is now the usual Euclidean cone over $X^\infty$, i.e.\ $CX^\infty$ is the quotient of $X^\infty\times [0,\infty)$ under the equivalence relation $(\xi,s)\sim (\xi',s')$ if, and only if, $s=s'=0$ or $(\xi,s)=(\xi',s')$. We denote an equivalence class of $(\xi,s)\in X^\infty\times [0,\infty)$ in $CX^\infty$ by $[\xi,s]$, and write $[0]$ for the equivalence class of $(\xi,0)$ for some/any $\xi\in X^\infty$. We refer to Chapter II.8 in \cite{BridsonHaefliger1999} for further information on the boundary $X^\infty$ and the boundary cone $CX^\infty$.\footnote{Topologically, the spaces $X^\infty$ and $CX^\infty$ actually require a more subtle treatment via the so-called cone topology (see Definition II.8.6 in \cite{BridsonHaefliger1999} and the discussion in \cite{GoodwinLewisLopezAcedoNicolae2024}), but this will not explicitly feature in the following arguments and we hence omit it.}

Following \cite{GoodwinLewisLopezAcedoNicolae2024}, we define the ``pairing'' function $\langle\cdot,\cdot\rangle:X\times CX^\infty\to\mathbb{R}$ by
\[
\langle x,[\xi,s]\rangle:=\begin{cases}sb_{\xi}(x)&\text{if }s>0,\\0&\text{if }s=0,\end{cases}
\]
where we wrote 
\[
b_{\xi}(x):=\lim_{t\to\infty} (d(x,r_{\overline{x},\xi}(t))-t)
\]
for the Busemann function $X\to\mathbb{R}$ (see e.g.\ Example 2.2.10 in \cite{Bacak2014a}) corresponding to the (unique) ray $r_{\overline{x},\xi}$ with direction $\xi$ and some arbitrary but fixed origin $\overline{x}$. Given a set $C\subseteq X$ and a function $g:C\to\mathbb{R}$, a Busemann subgradient of $g$ at $x\in C$ is then a point $[\xi,s]\in CX^\infty$ such that 
\[
x=\mathrm{argmin}_{y\in C}\left\{g(y)-\langle y,[\xi,s]\rangle\right\}.
\]
$g$ is called Busemann subdifferentiable if $g$ has a Busemann subgradient at every $x\in C$. 

We are now in the position to introduce the related projected Busemann subgradient method. For that, let $(\Omega,\mathsf{F},\PP)$ and $(E,\mathsf{E},\mu)$ be probability spaces, with $(E,\mathsf{E},\mu)$ complete, and let $X$ be a separable Hadamard space with the geodesic extension property and at least two points.\footnote{See \cite{GoodwinLewisLopezAcedoNicolae2024} for a discussion of this assumption, which in particular entails that the boundary $X^\infty$ is non-empty.} Further, fix a closed convex non-empty subset $C\subseteq X$ and let $f:E\times C\to \mathbb{R}$ be a given functional. As before, setting $\underline{f}(x):=\int f(e,x)\,d\mu(e)$, we want to
\[
\mbox{find some element of }\mathrm{argmin} \underline{f},
\]
where we again assume that $\underline{f}$ is proper and that $\mathrm{argmin}\underline{f}\neq\emptyset$, with $\mathrm{argmin} \underline{f}$ now understood to be a subset of $C$. We again set $F(x):=\underline{f}(x)-\min\underline{f}$.

In terms of assumption on $f$, we have to make three distinct types of assumptions, all relating to the Busemann subgradient structure of $f$. Concretely, in terms of basic regularity of the problem, we assume
\[
\begin{cases}f(e,\cdot)\text{ is Busemann subdifferentiable for any }e\in E\\
\text{and }f(\cdot,x)\text{ is measurable for all }x\in C.\end{cases}\tag{\textsf{SB}-\textsf{A1}}\label{SBA1}
\]
Further, we make the following Lipschitz-type assumption:
\[
\begin{cases}\text{There is a positive } L\in L^2(E,\mu)\text{ with $\int L\,d\mu>0$ and for any }e\in E\text{ and any}\\
\text{Busemann subgradient }[\xi,s]\in CX^\infty\text{ of }f(e,\cdot)\text{ at }x\in C\text{, we have }s\leq L(e).
\end{cases}\tag{\textsf{SB}-\textsf{A2}}\label{SBA2}
\]
Lastly, we want to be able to draw subgradients in measurable way, and to that end assume that there exists an oracle function $\mathsf{Busemann}_f$ which associates a function $\mathsf{Busemann}_f(\zeta,x):\Omega\to CX^\infty$ to input functions $\zeta:\Omega\to E$ and $x:\Omega\to C$ in such a way that it almost surely selects Busemann subgradients and preserves measurability:
\[
\begin{cases}\text{If }x\text{ and }\zeta\text{ are measurable, then }\mathsf{Busemann}_f(\zeta,x)\text{ is $\sigma(\zeta,x)$-measurable}\\
\text{and almost surely represents a Busemann subgradient of }f(\zeta,\cdot)\text{ at }x.
\end{cases}\tag{\textsf{SB}-\textsf{A3}}\label{SBA3}
\]
\eqref{SBA1} and \eqref{SBA2} in particular imply that $f(e,\cdot)$ is convex and $L(e)$-Lipschitz on $C$ (see \cite{GoodwinLewisLopezAcedoNicolae2024}). Hence, they guarantee that $F$ is measurable and that $\mathrm{argmin} \underline{f}=\mathrm{zer}F$ is closed, similar to before. \eqref{SBA3} does not explicitly feature in \cite{GoodwinLewisLopezAcedoNicolae2024}, but is crucial to guarantee the measurability of the associated subgradient method (see the discussion in \cite{Pischke2026}). In particular, an oracle satisfying \eqref{SBA3} always exists over separable Hadamard spaces (see again \cite{Pischke2026}).

Extending \cite{GoodwinLewisLopezAcedoNicolae2024} as in \cite{Pischke2026}, we now consider the iteration
\[
x_{n+1}:=P_C(r_{x_n,\xi_{n}}(s_nt_n))\tag{\textsf{SB}}\label{SB}
\]
where $[\xi_n,s_n]=\mathsf{Busemann}_f(\zeta_{n+1},x_{n})$, given a starting point $x_0\in C$ and sequences $(t_n)$ of positive reals as well as $(\zeta_{n+1})$ of random variables $\Omega\to E$, for which we assume that 
\[
(\zeta_{n+1})\text{ are i.i.d.\ with distribution $\mu$ and }\sum_{n\in\mathbb{N}}t_n=\infty, \sum_{n\in\mathbb{N}}t_n^2<\infty.\tag{\textsf{SB}-\textsf{A4}}\label{Bparameters}
\]
The conditions \eqref{SBA1} -- \eqref{Bparameters} together imply that $x_n$ is $\sigma(\zeta_1,\dots,\zeta_n)$-measurable for any $n\in\mathbb{N}$ (see Lemma 5.5 in \cite{Pischke2026}).

We here present the following general result on the effective convergence of \eqref{SB} under a stochastic regularity assumption:

\begin{theorem}\label{SBrates}
Let $(E,\mathsf{E},\mu)$ and $(\Omega,\mathsf{F},\PP)$ be probability spaces, with $(E,\mathsf{E},\mu)$ complete, let $X$ be a separable Hadamard space with the geodesic extension property and at least two points, and fix a closed convex non-empty subset $C\subseteq X$. Let $f:E\times C\to \mathbb{R}$ be a function with properties  \eqref{SBA1} -- \eqref{SBA3}  as above and assume additionally that $\underline{f}(x):=\int f(e,x)\,d\mu(e)$ is proper and $\mathrm{argmin}\underline{f}\neq\emptyset$. Write $F(x):=\underline{f}(x)-\min\underline{f}$. Let $(x_n)$ be the iteration given by \eqref{SB}, and assume \eqref{Bparameters}. Lastly, let $\tau:(0,\infty)\to (0,\infty)$ be a modulus of regularity for $F$ in mean w.r.t.\ $D$, i.e.
\[
\forall \varepsilon>0\ \forall x\in D\left(\EE[\underline{f}(x)-\min\underline{f}]<\tau(\varepsilon)\to \EE[\mathrm{dist}^2_{\mathrm{argmin}\underline{f}}(x)]<\varepsilon\right),
\]
where $D$ is a collection of $X$-valued random variables with $(x_n)\subseteq D$. Then $(x_n)$ strongly converges to an $\mathrm{argmin}\underline{f}$-valued random variable $x$ a.s.\ and in mean. Moreover, the following rates of convergence apply: Let $z\in\mathrm{argmin}\underline{f}$ and let $b>d(x_0,z)$. Write $\underline{L}:=\int L^2\,d\mu$. Further, let $\chi:(0,\infty)\to\mathbb{N}$ and $\theta:\mathbb{N}\times (0,\infty)\to\mathbb{N}$ be such that $\sum_{n=\chi(\varepsilon)}^\infty t_n^2<\varepsilon$ for all $\varepsilon>0$ and $\sum_{n=k}^{\theta(k,b)}t_n\geq b$ for all $b>0$ and $k\in\mathbb{N}$, and let $T > \sum_{n=0}^\infty t_n^2$. Then $\EE[d(x_n,x)]\to 0$ with rate $\rho(\varepsilon^2/4)$ and $d(x_n,x)\to 0$ a.s.\ with rate $\rho(\lambda\varepsilon^2/4)$, where
\[
\rho(\varepsilon):=\theta\left(\chi\left(\frac{\varepsilon}{6\underline{L}}\right),\frac{b+\underline{L}T}{\tau\left(\varepsilon/6\right)}\right).
\]
\end{theorem}

As before, the regularity assumption featuring in Theorem \ref{SBrates} above can be concretely instantiated in various ways (recall Section \ref{sec:examples}). Already for regularity notions such as (special cases of) weak sharp minima or error bounds, the a.s.\ strong convergence of the iteration outside of locally compact Hadamard spaces seems to be novel to the literature and the following corollary explicitly spells this out for the special case of quadratic growth, similar to before. We also note that by moving to a finite measure space $(E,\mathsf{E},\mu)$, our result in particular also applies to the method studied in \cite{GoodwinLewisLopezAcedoNicolae2024} (see \cite{Pischke2026} for a comparison). As before, using linear regularity assumptions and suitable conditions on the parameters we can apply our result on fast rates (recall Theorem \ref{thm:fast}) to obtain linear non-asymptotic guarantees.

\begin{corollary}
Let $(E,\mathsf{E},\mu)$ and $(\Omega,\mathsf{F},\PP)$ be probability spaces, with $(E,\mathsf{E},\mu)$ complete, let $X$ be a separable Hadamard space with the geodesic extension property and at least two points, and fix a closed convex non-empty subset $C\subseteq X$. Let $f:E\times C\to \mathbb{R}$ be a function with properties  \eqref{SBA1} -- \eqref{SBA3}  as above and assume additionally that $\underline{f}(x):=\int f(e,x)\,d\mu(e)$ is proper and $\mathrm{argmin}\underline{f}\neq\emptyset$. Let $(x_n)$ be the iteration given by \eqref{SB}, and assume \eqref{Bparameters}. Assume that $\underline{f}$ satisfies the quadratic growth condition
\[
c(\underline{f}(x)- \min \underline{f})\geq (\mathrm{dist}_{\mathrm{argmin}\underline{f}}(x))^2\text{ for all }x\in X,
\]
for some $c>0$. Then $(x_n)$ strongly converges to an $\mathrm{argmin}\underline{f}$-valued random variable $x$ a.s.\ and in mean. Moreover, the following rates of convergence apply: For $b$, $\underline{L}$, $\chi$, $\theta$ and $T$ as in Theorem \ref{SBrates}, we have $\EE[d(x_n,x)]\to 0$ with rate $\rho(\varepsilon^2/4)$ and $d(x_n,x)\to 0$ a.s.\ with rate $\rho(\lambda\varepsilon^2/4)$, where
\[
\rho(\varepsilon):=\theta\left(\chi\left(\frac{\varepsilon}{6\underline{L}}\right),\frac{c(b+\underline{L}T)}{\varepsilon/6}\right).
\]
\end{corollary}

The above corollary in particular includes the quantitative result given in \cite{Pischke2026} under the assumption that $f(e,\cdot)$ is strongly convex with parameter $\alpha(e)>0$ for every $e\in E$ such that $\underline{\alpha}:=\int \alpha\,d\mu>0$ (where we can instantiate the above with $c:=\frac{8}{\underline{\alpha}}$ as before). Again, by virtue of \eqref{CN}, a particular (class of) function(s) that satisfies the strong convexity requirement is the square metric $d^2$ with $c:=4$, so that also the above result in particular applies to the instantiation of \eqref{SB} applied for finding Fr\'echet means (see again \cite{GoodwinLewisLopezAcedoNicolae2024,Pischke2026}).

The proof of Theorem \ref{SBrates} proceeds similar to both subsections before, and in particular relies on establishing the strong stochastic quasi-Fej\'er monotonicity of the iteration together with a $\liminf$-bound. For that, we introduce some notation, similar to before. Define the filtration $\mathsf{F}_n:=\sigma(\zeta_1,\dots,\zeta_n)$ and write $\EE_n$ for the conditional expectation $\mathbb{E}[\cdot \mid\mathsf{F}_n]$. As before, we write $\underline{L}:=\int L^2\,d\mu<\infty$.

As a last preliminary result, we require the following property of the Busemann subgradients:

\begin{lemma}[Lemma 6.1 in \cite{GoodwinLewisLopezAcedoNicolae2024}]\label{BSLemma}
Let $g:C\to \mathbb{R}$ be a given function for a non-empty closed and convex set $C\subseteq X$ and let $[\xi,s]$ be a Busemann subgradient of $f$ at $x\in C$. Given $t>0$, define
\[
x^+:=\begin{cases}P_C(r_{x,\xi}(st))&\text{if }s>0,\\x&\text{if }s=0.\end{cases}
\]
Then for any $y\in C$: $d^2(x^+,y)\leq d^2(x,y)-2t(f(x)-f(y))+s^2t^2$.
\end{lemma}

\begin{lemma}[extending \cite{GoodwinLewisLopezAcedoNicolae2024} and \cite{Pischke2026}]\label{BIneqMain}
Let $n\in\mathbb{N}$. Then for any $C$-valued $\mathsf{F_n}$-measurable random variable $y$:
\[
\EE_n[d^2(x_{n+1},y)]\leq d^2(x_n,y)-2t_n(\underline{f}(x_n)-\underline{f}(y))+\underline{L}t_n^2\text{ a.s.}
\]
In particular, if $y$ is additionally such that $y\in\mathrm{argmin}\underline{f}$ a.s., then
\[
\EE_n[d^2(x_{n+1},y)]\leq d^2(x_n,y)-2t_n(\underline{f}(x_n)-\min\underline{f})+\underline{L}t_n^2\text{ a.s.}
\]
\end{lemma}
\begin{proof}
Using Lemma \ref{BSLemma} together with assumption \eqref{SBA2}, by which we have $s_n\leq L(\zeta_{n+1})$, and applying conditional expectations yields
\[
\EE_n[d^2(x_{n+1},y)]\leq d^2(x_n,y)-2t_n(\EE_n[f(\zeta_{n+1},x_n)]-\EE_n[f(\zeta_{n+1},y)])+\underline{L}t_n^2,
\]
also using that $\zeta_{n+1}$ has distribution $\mu$ by assumption \eqref{Bparameters}. Now, as $\zeta_{n+1}$ is independent of $x_n$ and $\mathsf{F}_n$, again using assumption \eqref{Bparameters} as well as that $x_n$ is $\mathsf{F}_n$-measurable, we have 
\[
\EE_n[f(\zeta_{n+1},x_n)](\omega)=\EE_{\omega'\sim\PP}[ f(\zeta_{n+1}(\omega'),x_n(\omega))]=\underline{f}(x_n(\omega))
\]
and similarly $\EE_n[f(\zeta_{n+1},y)]=\underline{f}(y)$. This yields the claim.
\end{proof}

\begin{lemma}\label{SBliminf}
Let $z\in\mathrm{argmin}\underline{f}$ and let $b>d(x_0,z)$. Let $\theta:\mathbb{N}\times (0,\infty)\to\mathbb{N}$ be such that $\sum_{n=k}^{\theta(k,b)}t_n\geq b$ for all $b>0$ and $k\in\mathbb{N}$, and let $T > \sum_{n=0}^\infty t_n^2$. Then we have
\[
\forall \varepsilon>0\ \forall N\in\mathbb{N}\ \exists n\in [N;\varphi(\varepsilon,N)]\left( \EE[F(x_n)]<\varepsilon\right)
\]
where $\varphi(\varepsilon,N):=\theta(N,(b+\underline{L}T)/\varepsilon)$.
\end{lemma} 
\begin{proof}
Taking expectations in Lemma \ref{BIneqMain}, applied to $z$, and applying Lemma \ref{qihou} yields the bound $\sum_{n=0}^\infty t_n\EE[F(x_n)-\min F]<b+\underline{L}T$ and so Lemma \ref{sumconv} yields the claim.
\end{proof}

\begin{proof}[Proof of Theorem \ref{SBrates}]
Note that $\chi(\varepsilon/\underline{L})$ is a rate of convergence for $\sum_{n\in\mathbb{N}}\underline{L}t_n^2<\infty$ as we have $\sum_{n=\chi(\varepsilon/\underline{L})}^\infty t_n^2<\frac{\varepsilon}{\underline{L}}$ and so $\sum_{n=\chi(\varepsilon/\underline{L})}^\infty \underline{L}\lambda_n^2<\varepsilon$ for all $\varepsilon>0$. As before, the result then follows from Theorem \ref{thm:main} using Lemmas \ref{BIneqMain} and \ref{SBliminf} and that $d^2$ is uniformly consistent with modulus $\theta(\varepsilon):=\varepsilon^2$ (recall Example \ref{ex:distances}).
\end{proof}


\begin{thebibliography}{10}

\bibitem{Aleksandrov1951}
A.D. Aleksandrov.
\newblock {A theorem on triangles in a metric space and some of its
  applications}.
\newblock {\em Trudy Matematicheskogo Instituta imeni V.A. Steklova}, 38:5--23,
  1951.

\bibitem{AlexanderKapovitchPetrunin2023}
S.~Alexander, V.~Kapovitch, and A.~Petrunin.
\newblock {\em {Alexandrov Geometry: Foundations}}, volume 236 of {\em Graduate
  Studies in Mathematics}.
\newblock American Mathematical Society, Providence, RI, 2024.

\bibitem{AliprantisBorder2006}
C.D. Aliprantis and K.C. Border.
\newblock {\em {Infinite Dimensional Analysis: A Hitchhiker's Guide}}.
\newblock Springer Berlin, Heidelberg, 2006.

\bibitem{AsiChadhaChengDuchi2020}
H.~Asi, K.~Chadha, G.~Cheng, and J.C. Duchi.
\newblock {Minibatch stochastic approximate proximal point methods}.
\newblock In {\em {Proceedings of the 34th International Conference on Neural
  Information Processing Systems (NeurIPS 2020)}}, pages 21958--21968. Curran
  Associates Inc., USA, 2020.

\bibitem{AsiDuchi2019}
H.~Asi and J.C. Duchi.
\newblock {Stochastic (approximate) proximal point methods: Convergence,
  optimality, and adaptivity}.
\newblock {\em SIAM Journal on Optimization}, 29(3):2257--2290, 2019.

\bibitem{AubinFrankowska2009}
J.-P. Aubin and H.~Frankowska.
\newblock {\em {Set-Valued Analysis}}.
\newblock Springer, New York, 2009.

\bibitem{Aumann1969}
R.J. Aumann.
\newblock {Measurable utility and the measurable choice problem}.
\newblock In G.T. Guilbaud, editor, {\em La D\'ecision}, pages 15--26. Editions
  du Centre National de la Recherche Scientifique, Paris, 1969.

\bibitem{BauschkeCombettes2017}
H.H. Bauschke and P.L. Combettes.
\newblock {\em {Convex Analysis and Monotone Operator Theory in Hilbert
  Spaces}}.
\newblock CMS Books in Mathematics. Springer Cham, 2nd edition, 2017.

\bibitem{BauschkeLewis2000}
H.H. Bauschke and A.S. Lewis.
\newblock {Dykstras algorithm with Bregman projections: A convergence proof}.
\newblock {\em Optimization}, 48(4):409--427, 2000.

\bibitem{Bacak2013}
M.~Ba\v{c}\'ak.
\newblock {The proximal point algorithm in metric spaces}.
\newblock {\em Israel Journal of Mathematics}, 194(2):689--701, 2013.

\bibitem{Bacak2014b}
M.~Ba\v{c}\'ak.
\newblock {Computing medians and means in Hadamard spaces}.
\newblock {\em SIAM Journal of Optimization}, 24(3):1542--1566, 2014.

\bibitem{Bacak2014a}
M.~Ba\v{c}\'ak.
\newblock {\em {Convex analysis and optimization in Hadamard spaces}},
  volume~22 of {\em De Gruyter Series in Nonlinear Analysis and Applications}.
\newblock Walter de Gruyter GmbH, Berlin/Boston, 2014.

\bibitem{Bacak2018}
M.~Ba\v{c}\'ak.
\newblock {A variational approach to stochastic minimization of convex
  functionals}.
\newblock {\em Pure and Applied Functional Analysis}, 3(2):287--295, 2018.

\bibitem{Bertsekas2011}
D.P. Bertsekas.
\newblock {Incremental proximal methods for large scale convex optimization}.
\newblock {\em Mathematical Programming}, 129:163--195, 2011.

\bibitem{Bertsekas2012}
D.P. Bertsekas.
\newblock {Incremental gradient, subgradient, and proximal methods for convex
  optimization: A survey}.
\newblock In S.~Sra, S.~Nowozin, and S.J. Wright, editors, {\em {Optimization
  for Machine Learning}}, Neural Information Processing Series, pages 85--120.
  The MIT Press, Cambridge, Massachusetts, 2012.

\bibitem{Bianchi2016}
P.~Bianchi.
\newblock {Ergodic convergence of a stochastic proximal point algorithm}.
\newblock {\em SIAM Journal on Optimization}, 26(4):2235--2260, 2016.

\bibitem{BilleraHolmesVogtmann2001}
L.J. Billera, S.P. Holmes, and K.~Vogtmann.
\newblock {Geometry of the space of phylogenetic trees}.
\newblock {\em Advances in Applied Mathematics}, 27(4):733--767, 2001.

\bibitem{BotSchindler2024}
R.I. Bo\c{t} and C.~Schindler.
\newblock {On a stochastic differential equation with correction term governed
  by a monotone and Lipschitz continuous operator}.
\newblock {\em Evolution Equations and Control Theory}, 14(3):463--493, 2025.

\bibitem{BotSchindler2025}
R.I. Bo\c{t} and C.~Schindler.
\newblock {Long-Time Analysis of Stochastic Heavy Ball Dynamics for Convex
  Optimization and Monotone Equations}.
\newblock {\em Discrete and Continuous Dynamical Systems}, 51:439--474, 2026.

\bibitem{BolteDaniilidisLeyMazet2010}
J.~Bolte, A.~Daniilidis, O.~Ley, and L.~Mazet.
\newblock {Characterizations of {\L}ojasiewicz inequalities: subgradient flows,
  talweg, convexity}.
\newblock {\em Transactions of the American Mathematical Society},
  362(6):3319--3363, 2010.

\bibitem{BolteNguyenPeypouquetSuter2017}
J.~Bolte, T.P. Nguyen, J.~Peypouquet, and B.W. Suter.
\newblock {From error bounds to the complexity of first-order descent methods
  for convex functions}.
\newblock {\em Mathematical Programming}, 165:471--507, 2017.

\bibitem{BorweinLiTam2017}
J.M. Borwein, G.~Li, and M.K. Tam.
\newblock {Convergence rate analysis for averaged fixed point iterations in
  common fixed point problems}.
\newblock {\em SIAM Journal on Optimization}, 27:1--33, 2017.

\bibitem{BrezisLions1978}
H.~Br\'ezis and P.L. Lions.
\newblock {Produits infinis de resolvantes}.
\newblock {\em Israel Journal of Mathematics}, 29(4):329--345, 1978.

\bibitem{BridsonHaefliger1999}
M.R. Bridson and A.~Haefliger.
\newblock {\em {Metric Spaces of Non-Positive Curvature}}, volume 319 of {\em
  Grundlehren der mathematischen Wissenschaften}.
\newblock Springer Berlin, Heidelberg, 1999.

\bibitem{BruhatTits1972}
F.~Bruhat and J.~Tits.
\newblock {Groupes r\'eductifs sur un corps local. I. Donn\'ees radicielles
  valu\'ees}.
\newblock {\em Publications Math\'ematiques de l'Institut des Hautes \'Etudes
  Scientifiques}, 41:5--251, 1972.

\bibitem{BurkeDeng2002}
J.V. Burke and S.~Deng.
\newblock {Weak sharp minima revisited. I. Basic theory}.
\newblock {\em Control and Cybernetics}, 31:439--469, 2002.

\bibitem{BurkeFerris1993}
J.V. Burke and M.C. Ferris.
\newblock {Weak sharp minima in mathematical programming}.
\newblock {\em SIAM Journal on Control and Optimization}, 31:1340--1359, 1993.

\bibitem{ButnariuIusem2000}
D.~Butnariu and A.N. Iusem.
\newblock {\em {Totally Convex Functions for Fixed Points Computation and
  Infinite Dimensional Optimization}}, volume~40 of {\em Applied Optimization}.
\newblock Springer Dordrecht, 2000.

\bibitem{ButnariuResmerita2006}
D.~Butnariu and E.~Resmerita.
\newblock {Bregman distances, totally convex functions and a method for solving
  operator equations in Banach spaces}.
\newblock {\em Abstract and Applied Analysis}, 2006.
\newblock Article no.\ 84919, 39pp.

\bibitem{CastaingValadier1977}
C.~Castaing and M.~Valadier.
\newblock {\em {Convex Analysis and Measurable Multifunctions}}, volume 580 of
  {\em Lecture Notes in Mathematics}.
\newblock Springer Berlin, Heidelberg, 1977.

\bibitem{ChaipunyaKohsakaKumam2021}
P.~Chaipunya, F.~Kohsaka, and P.~Kumam.
\newblock {Monotone vector fields and generation of nonexpansive semigroups in
  complete $\mathrm{CAT}(0)$ spaces}.
\newblock {\em Numerical Functional Analysis and Optimization},
  42(9):989--1018, 2021.

\bibitem{Combettes2001}
P.L. Combettes.
\newblock {Quasi-Fej\'erian Analysis of Some Optimization Algorithms}.
\newblock In D.~Butnariu, Y.~Censor, and S.~Reich, editors, {\em {Inherently
  Parallel Algorithms in Feasibility and Optimization and Their Applications}},
  volume~8 of {\em Studies in Computational Mathematics}, pages 115--152.
  North-Holland, Amsterdam, 2001.

\bibitem{Combettes2009}
P.L. Combettes.
\newblock {Fej\'er monotonicity in convex optimization}.
\newblock In C.A. Floudas and P.M. Pardalos, editors, {\em {Encyclopedia of
  Optimization}}, pages 1016--1024. Springer, Boston, MA, 2009.

\bibitem{CombettesMadariaga2025}
P.L. Combettes and J.I. Madariaga.
\newblock {A geometric framework for stochastic iterations}.
\newblock {\em Mathematics of Computation}, 2026.
\newblock To appear.

\bibitem{CombettesPesquet2015}
P.L. Combettes and J.C. Pesquet.
\newblock {Stochastic quasi-Fej\'er block-coordinate fixed point iterations
  with random sweeping}.
\newblock {\em SIAM Journal on Optimization}, 25(2):1221--1248, 2015.

\bibitem{CombettesPesquet2019}
P.L. Combettes and J.C. Pesquet.
\newblock {Stochastic quasi-Fej\'er block-coordinate fixed point iterations
  with random sweeping II: mean-square and linear convergence}.
\newblock {\em Mathematical Programming}, 174(1):433--451, 2019.

\bibitem{Cooper2004}
S.B. Cooper.
\newblock {\em {Computability Theory}}.
\newblock Chapman and Hall/CRC, New York, 2004.

\bibitem{DontchevRockafellar2009}
A.L. Dontchev and R.T. Rockafellar.
\newblock {\em {Implicit functions and solution mappings. A view from
  variational analysis}}.
\newblock Springer Monographs in Mathematics. Springer, Dordrecht, 2009.

\bibitem{DrusvyatskiyLewis2018}
D.~Drusvyatskiy and A.S. Lewis.
\newblock {Error bounds, quadratic growth, and linear convergence of proximal
  methods}.
\newblock {\em Mathematics of Operations Research}, 43(3):919--948, 2018.

\bibitem{Ermolev1969}
Y.M. Ermol'ev.
\newblock {On the method of generalized stochastic gradients and quasi-Fej\'er
  sequences}.
\newblock {\em Cybernetics}, 5:208--220, 1969.

\bibitem{Ermolev1971}
Y.M. Ermol'ev.
\newblock {On convergence of random quasi-Fej\'er sequences}.
\newblock {\em Cybernetics}, 7:655--656, 1971.

\bibitem{ErmolevTuniev1973}
Y.M. Ermol'ev and A.D. Tuniev.
\newblock Random fej\'er and quasi-fej\'er sequences.
\newblock {\em Theory of Optimal Solutions -- Akademiya Nauk Ukrainsko\u{i},
  SSR Kiev}, 2:76--83, 1968.
\newblock English translation in Amer. Math. Soc. Select. Translat. Math.
  Statist. Probab., 13 (1973), pp. 143--148.

\bibitem{Ferris1991}
M.C. Ferris.
\newblock {Finite termination of the proximal point algorithm}.
\newblock {\em Mathematical Programming}, 50:359--366, 1991.

\bibitem{FreundPischke2026}
A.~Freund and N.~Pischke.
\newblock {Effective rates for continuous-time quasi-Fej\'er monotone dynamical
  systems}, 2026.
\newblock Preprint, available at \url{https://arxiv.org/abs/2603.23708}.

\bibitem{GoodwinLewisLopezAcedoNicolae2024}
A.~Goodwin, A.S. Lewis, G.~L\'opez-Acedo, and A.~Nicolae.
\newblock {Stochastic and incremental subgradient methods for convex
  optimization on Hadamard spaces}.
\newblock {\em Mathematical Programming}, 2026.
\newblock To appear.

\bibitem{Grasmair2010}
M.~Grasmair.
\newblock {Generalized Bregman distances and convergence rates for non-convex
  regularization methods}.
\newblock {\em Inverse Problems}, 26(11), 2010.
\newblock Article no.\ 115014, 16pp.

\bibitem{Gromov1987}
M.~Gromov.
\newblock {Hyperbolic groups}.
\newblock In S.M. Gersten, editor, {\em {Essays in group theory}}, volume~8 of
  {\em Mathematical Sciences Research Institute Publications}, pages 75--263.
  Springer, New York, 1987.

\bibitem{Gueler1991}
O.~G\"uler.
\newblock {On the convergence of the proximal point algorithm for convex
  minimization}.
\newblock {\em SIAM Journal on Control and Optimization}, 29:403--419, 1991.

\bibitem{HermerLukeSturm2019}
N.~Hermer, D.R. Luke, and A.~Sturm.
\newblock {Random function iterations for consistent stochastic feasibility}.
\newblock {\em Numerical Functional Analysis and Optimization}, 40(4):386--420,
  2019.

\bibitem{Jost1995}
J.~Jost.
\newblock {Convex functionals and generalized harmonic maps into spaces of
  nonpositive curvature}.
\newblock {\em Commentarii Mathematici Helvetici}, 70:659--673, 1995.

\bibitem{Klenke2020}
A.~Klenke.
\newblock {\em {Probability Theory: A Comprehensive Course}}.
\newblock Universitext. Springer Cham, 3rd edition, 2020.

\bibitem{Kohlenbach2008}
U.~Kohlenbach.
\newblock {\em {Applied Proof Theory: Proof Interpretations and their Use in
  Mathematics}}.
\newblock Springer Monographs in Mathematics. Springer-Verlag Berlin
  Heidelberg, 2008.

\bibitem{Kohlenbach2019}
U.~Kohlenbach.
\newblock {Proof-theoretic Methods in Nonlinear Analysis}.
\newblock In B.~Sirakov, P.~Ney de~Souza, and M.~Viana, editors, {\em
  {Proceedings of ICM 2018}}, volume~2, pages 61--82. World Scientific,
  Singapure, 2019.

\bibitem{KohlenbachLopezAcedoNicolae2019}
U.~Kohlenbach, G.~L\'opez-Acedo, and A.~Nicolae.
\newblock {Moduli of regularity and rates of convergence for Fej\'er monotone
  sequences}.
\newblock {\em Israel Journal of Mathematics}, 232:261--297, 2019.

\bibitem{Krasnoselskii1955}
M.A. Krasnoselskii.
\newblock {Two remarks on the method of successive approximations}.
\newblock {\em Uspekhi Matematicheskikh Nauk}, 10(1(63)):123--127, 1955.

\bibitem{LeusteanSipos2018}
L.~Leu\c{s}tean and A.~Sipo\c{s}.
\newblock {Effective strong convergence of the proximal point algorithm in
  CAT(0) spaces}.
\newblock {\em Journal of Nonlinear and Variational Analysis}, 2(2):219--228,
  2018.

\bibitem{Leventhal2009}
D.~Leventhal.
\newblock {Metric subregularity and the proximal point method}.
\newblock {\em Journal of Mathematical Analysis and Applications},
  360:681--688, 2009.

\bibitem{LiLopezMartinMarquez2009}
C.~Li, G.~L\'opez, and V.~Mart\'in-M\'arquez.
\newblock {Monotone vector fields and the proximal point algorithm on Hadamard
  manifolds}.
\newblock {\em Journal of the London Mathematical Society}, 79(3):663--683,
  2009.

\bibitem{LiMordukhovichWangYao2011}
C.~Li, B.S. Mordukhovich, J.H. Wang, and J.C. Yao.
\newblock {Weak sharp minima on Riemannian manifolds}.
\newblock {\em SIAM Journal on Optimization}, 21:1523--1560, 2011.

\bibitem{LiMordukhovichZhu2026}
G.~Li, B.~Mordukhovich, and J.~Zhu.
\newblock {Generalized metric subregularity with applications to high-order
  regularized Newton methods}.
\newblock {\em Mathematics of Operations Research}, 2026.
\newblock To appear.

\bibitem{LiuLongLiHuang2025}
J.~Liu, X.J. Long, X.S. Li, and N.J. Huang.
\newblock {Stochastic dual dynamical systems for linear equality constrained
  convex optimization problems}.
\newblock {\em Communications in Nonlinear Science and Numerical Simulation},
  154, 2026.
\newblock Article no.\ 109538, 15pp.

\bibitem{LukeSchnebelStaudiglPeypouquetQu2026}
D.R. Luke, J.C. Schnebel, M.~Staudigl, J.~Peypouquet, and S.~Qu.
\newblock {Asymptotic behaviour of coupled random dynamical systems with
  multiscale aspects}, 2026.
\newblock Preprint, available at \url{https://arxiv.org/abs/2601.15411}.

\bibitem{Mann1953}
W.R. Mann.
\newblock {Mean value methods in iteration}.
\newblock {\em Proceedings of the American Mathematical Society}, 4:506--510,
  1953.

\bibitem{Martinet1970}
B.~Martinet.
\newblock {R\'egularisation din\'equations variationnelles par approximations
  successives}.
\newblock {\em Revue fran\c{c}aise d'informatique et de recherche
  op\'erationnelle}, 4:154--159, 1970.

\bibitem{MaulenSotoFadiliAttouch2025}
R.~Maulen-Soto, J.~Fadili, and H.~Attouch.
\newblock {An stochastic differential equation perspective on stochastic convex
  optimization}.
\newblock {\em Mathematics of Operations Research}, 50(4):3190--3221, 2025.

\bibitem{MaulenSotoFadiliAttouchOchs2026}
R.~Maulen-Soto, J.~Fadili, H.~Attouch, and P.~Ochs.
\newblock {Stochastic inertial dynamics via time scaling and averaging}.
\newblock {\em Stochastic Systems}, 16(1):61--89, 2026.

\bibitem{Mayer1998}
U.~Mayer.
\newblock {Gradient flows on nonpositively curved metric spaces and harmonic
  maps}.
\newblock {\em Communications in Analysis and Geometry}, 6:199--253, 1998.

\bibitem{MoultonSpillner2025}
V.~Moulton and A.~Spillner.
\newblock {Spaces of ranked tree-child networks}.
\newblock {\em Journal of Mathematical Biology}, 91(3), 2025.
\newblock Article no.\ 32, 26pp.

\bibitem{NemirovskiJuditskyLanShapiro2009}
A.~Nemirovski, A.~Juditsky, G.~Lan, and A.~Shapiro.
\newblock {Robust stochastic approximation approach to stochastic programming}.
\newblock {\em SIAM Journal of Optimization}, 19(4):1574--1609, 2009.

\bibitem{NeriPischkePowell2025}
M.~Neri, N.~Pischke, and T.~Powell.
\newblock {An abstract effective convergence theorem for stochastic processes,
  with applications to stochastic approximation}, 2026.
\newblock Preprint, available at \url{https://arxiv.org/abs/2504.12922}.

\bibitem{NeriPischkePowell2026}
M.~Neri, N.~Pischke, and T.~Powell.
\newblock {Generalized fluctuation bounds for stochastic algorithms in the
  presence of compactness}, 2026.
\newblock Preprint, available at \url{https://arxiv.org/abs/2602.22741}.

\bibitem{NeriPowell2024}
M.~Neri and T.~Powell.
\newblock {A quantitative {R}obbins-{S}iegmund theorem}.
\newblock {\em The Annals of Applied Probability}, 36(1):636--651, 2026.

\bibitem{Neumann2015}
E.~Neumann.
\newblock {Computational problems in metric fixed point theory and their
  Weihrauch degrees}.
\newblock {\em Logical Methods in Computer Science}, 11(4), 2015.
\newblock Article no.\ 20, 44pp.

\bibitem{OhtaPalfia2015}
S.I. Ohta and M.~P\'alfia.
\newblock {Discrete-time gradient flows and law of large numbers in Alexandrov
  spaces}.
\newblock {\em Calculus of Variations and Partial Differential Equations},
  54(2):1591--1610, 2015.

\bibitem{PintoPischke2026}
P.~Pinto and N.~Pischke.
\newblock {On Dykstra's algorithm with Bregman projections}, 2026.
\newblock Preprint, available at \url{https://nicholaspischke.github.io}.

\bibitem{Pischke2025b}
N.~Pischke.
\newblock {Generalized Fej\'er monotone sequences and their finitary content}.
\newblock {\em Optimization}, 74(14):3771--3838, 2025.

\bibitem{Pischke2025}
N.~Pischke.
\newblock {Mean-square and sublinear convergence of a stochastic proximal point
  algorithm in metric spaces of nonpositive curvature}, 2026.
\newblock Preprint, available at \url{https://arxiv.org/abs/2510.10697}.

\bibitem{Pischke2026}
N.~Pischke.
\newblock {On Busemann subgradient methods for stochastic minimization in
  Hadamard spaces}, 2026.
\newblock Preprint, available at \url{https://arxiv.org/abs/2602.08127}.

\bibitem{Pischke2026b}
N.~Pischke.
\newblock {Weak convergence of the stochastic proximal point method in metric
  spaces}, 2026.
\newblock Preprint, available at \url{https://arxiv.org/abs/2605.20805}.

\bibitem{PischkeKohlenbach2024}
N.~Pischke and U.~Kohlenbach.
\newblock {Effective rates for iterations involving Bregman strongly
  nonexpansive operators}.
\newblock {\em Set-Valued and Variational Analysis}, 32(4), 2024.
\newblock Article no.\ 33, 58pp.

\bibitem{PischkePowell2024}
N.~Pischke and T.~Powell.
\newblock {Asymptotic regularity of a generalised stochastic Halpern scheme}.
\newblock {\em Journal of Optimization Theory and Applications}, 210(1), 2026.
\newblock Article no.\ 3, 43pp.

\bibitem{Qihou2001}
L.~Qihou.
\newblock {Iteration sequences for asymptotically quasi-nonexpansive mappings
  with error member}.
\newblock {\em Journal of Mathematical Analysis and Applications}, 259:18--24,
  2001.

\bibitem{RobbinsSiegmund1971}
H.~Robbins and D.~Siegmund.
\newblock {A convergence theorem for non negative almost supermartingales and
  some applications}.
\newblock In J.S. Rustagi, editor, {\em Optimizing Methods in Statistics},
  pages 233--257. Academic Press, New York, 1971.

\bibitem{Rockafellar1971}
R.T. Rockafellar.
\newblock {Convex integral functionals and duality}.
\newblock In E.H. Zarantonello, editor, {\em {Contributions to Nonlinear
  Functional Analysis}}, pages 215--236. Academic Press, New York, 1971.

\bibitem{Rockafellar1976}
R.T. Rockafellar.
\newblock {Monotone operators and the proximal point algorithm}.
\newblock {\em SIAM Journal of Control and Optimization}, 14:877--898, 1976.

\bibitem{Roemisch1986}
W.~R\"omisch.
\newblock {On the Convergence of Measurable Selections and an Application to
  Approximations in Stochastic Optimization}.
\newblock {\em Zeitschrift f\"ur Analysis und ihre Anwendungen}, 5(3):277--288,
  1986.

\bibitem{RyuBoyd}
E.K. Ryu and S.~Boyd.
\newblock {Stochastic Proximal Iteration: A Non-Asymptotic Improvement upon
  Stochastic Gradient Descent}.
\newblock working draft, accessed 2025,
  \url{https://ernestryu.com/papers/spi.pdf}.

\bibitem{SchmidtLeRoux2013}
M.~Schmidt and N.~Le Roux.
\newblock {Fast Convergence of Stochastic Gradient Descent under a Strong
  Growth Condition}, 2013.
\newblock Preprint, available at \url{https://arxiv.org/abs/1308.6370}.

\bibitem{Shapiro1994}
A.~Shapiro.
\newblock {Quantitative stability in stochastic programming}.
\newblock {\em Mathematical Programming}, 67(1):99--108, 1994.

\bibitem{Specker1949}
E.~Specker.
\newblock {Nicht konstruktiv beweisbare S\"atze der Analysis}.
\newblock {\em The Journal of Symbolic Logic}, 14:145--208, 1949.

\bibitem{Zhang2020}
H.~Zhang.
\newblock {New analysis of linear convergence of gradient-type methods via
  unifying error bound conditions}.
\newblock {\em Mathematical Programming}, 180(1):371--416, 2020.

\bibitem{ZhangSra2016}
H.~Zhang and S.~Sra.
\newblock {First-order methods for geodesically convex optimization}.
\newblock In V.~Feldman, A.~Rakhlin, and O.~Shamir, editors, {\em Proceedings
  of the 29th Annual Conference on Learning Theory (COLT 2016)}, volume~49 of
  {\em Proceedings of Machine Learning Research}, pages 1617--1638. PMLR, 2016.

\bibitem{ZhangMiDuSunWangLiZhou2025}
J.~R. Zhang, X.~Mi, G.~Du, Q.~Sun, S.~Wang, J.~Li, and W.~Zhou.
\newblock {A universal Banach--Bregman framework for stochastic iterations:
  Unifying stochastic mirror descent, learning and LLM training}, 2025.
\newblock Preprint, available at \url{https://arxiv.org/abs/2509.14216}.

\end{thebibliography}

\end{document}